\documentclass[12pt]{article}
\usepackage[utf8]{inputenc}
\usepackage{amsmath,amsthm,amssymb, tikz-cd}
\usepackage{mathrsfs}
\usepackage[normalem]{ulem}
\DeclareMathOperator{\qdim}{qdim}
\DeclareMathOperator{\Irr}{Irr}

\DeclareMathOperator{\diag}{diag}
\DeclareMathOperator{\glob}{glob}
\begin{document}
\input amssym.def
\setcounter{equation}{0}
\newcommand{\wt}{{\rm wt}}
\newcommand{\pt}{{\operatorname{pt}}}
\newcommand{\rank}{{\operatorname{rank}}}
\newcommand{\Quad}{{\operatorname{Quad}}}
\newcommand{\spa}{\mbox{span}}
\newcommand{\Res}{\mbox{Res}}
\newcommand{\End}{\mbox{End}}
\newcommand{\Ind}{\mbox{Ind}}
\newcommand{\Hom}{\mbox{Hom}}
\newcommand{\Mod}{\mbox{Mod}}
\newcommand{\Rep}{\mbox{Rep}}
\newcommand{\m}{\mbox{mod}\ }
\renewcommand{\theequation}{\thesection.\arabic{equation}}
\numberwithin{equation}{section}
\def\lcm{\operatorname{lcm}}
\def\ord{\operatorname{ord}}
\def\FSexp{\operatorname{FSexp}}
\def \1{{\mathbf{1}}}
\def \b1{{Bbb 1}}
\def \End{{\rm End}}
\def \Aut{{\rm Aut}}
\def \Z{\mathbb Z}
\def \H{\mathbf{H}}
\def \M{\Bbb M}
\def \C{\mathbb C}
\def \BR{\mathbb R}
\def \Q{\mathbb Q}
\def \N{\mathbb N}
\def \ann{{\rm Ann}}
\def \<{\langle}
\def \o{\omega}
\def \O{\Omega}
\def \Or{{\cal O}}
\def \M{{\cal M}}
\def \1t{\frac{1}{T}}
\def \>{\rangle}
\def \t{\tau }
\def \a{\alpha }
\def \w{\omega}
\def \e{\epsilon }
\def \l{\lambda }
\def \L{\Lambda }
\def \g{\gamma}
\def \G{\Gamma}
\def \b{\beta }
\def \om{\omega }
\def \o{\omega }
\def \ot{\otimes}
\def \cg{\chi_g}
\def \ag{\alpha_g}
\def \ah{\alpha_h}
\def \ph{\psi_h}
\def \S{\cal S}
\def \Cc{\cal C}
\def \Bb{\cal B}
\def \Aa{\cal A}
\def \ZZ{\cal Z}
\def \Dd{\cal D}
\def \nor{\vartriangleleft}
\def \V{V^{\natural}}
\def\voa{vertex operator algebra\ }
\def \vosa{vertex operator superalgebra\ }
\def \vosas{vertex operator superalgebras\ }
\def \voas{vertex operator algebras}
\def \v{vertex operator algebra\ }
\def \1{{\bf 1}}
\def \be{\begin{equation}\label}
\def \ee{\end{equation}}
\def \qed{\mbox{ $\square$}}
\def \pf {\noindent {\bf Proof:} \,}
\def \bl{\begin{lem}\label}
\def \el{\end{lem}}
\def \ba{\begin{array}}
\def \ea{\end{array}}
\def \bt{\begin{thm}\label}
\def \et{\end{thm}}
\def \br{\begin{rem}\label}
\def \er{\end{rem}}
\def \ed{\end{de}}
\def \bp{\begin{prop}\label}
\def \ep{\end{prop}}
\def \p{\phi}
\def \d{\delta}
\def \Vec{{\rm Vec}}
\def \Irr{{\rm Irr}}
\def \irr{{\rm Irr}}\def \glob{{\rm glob}}
\def \obj{{\rm obj}}
\def \Id{{\rm Id}}
\def\OO{\mathcal{O}}
\def\EE{{\mathcal{E}}}
\def\Vl0{{(V^l)_{\bar 0}}}
\def\Vm0{{(V^m)_{\bar 0}}}
\def\V0{{V_{\bar 0}}}
\def\VlZ{{V(l,\BZ)}}
\def\VlZpm{{V_{\pm}(l,\BZ) }}
\def\VlhZ{{V(l,\BZ+\frac{1}{2})}}
\newtheorem{th1}{Theorem}
\newtheorem{ree}[th1]{Remark}
\newtheorem{thm}{Theorem}[section]
\newtheorem{prop}[thm]{Proposition}
\newtheorem{coro}[thm]{Corollary}
\newtheorem{lem}[thm]{Lemma}
\newtheorem{rem}[thm]{Remark}
\newtheorem{de}[thm]{Definition}
\newtheorem{hy}[thm]{Hypothesis}
\newtheorem{conj}[thm]{Conjecture}
\newtheorem{ex}[thm]{Example}
\newtheorem{q}[thm]{Question}

\newcommand\red[1]{{\color{red} #1}}
\newcommand\blue[1]{{\color{blue} #1}}
\def\b1{\mathbf{1}}
\def\BZ{\mathbb{Z}}
\def\ol{\overline}
\def\id{\operatorname{id}}
\def\sV{\operatorname{sVec}}
\def\s{\sigma}
\def\ft{\frak t}
\def\i{\iota}
\def\uc{\underline{c}}
\def\CC{\mathcal{C}}
\def\FF{\mathcal{F}}
\def\SC{\mathcal{SC}}
\def\uSC{\underline{\SC}}
\def\DD{\mathcal{D}}
\def\BB{\mathcal{B}}
\def\ZZ{\mathcal{Z}}
\def\EE{\mathcal{E}}
\def\CV0{\mathcal{C}_{V_{\bar 0}}}
\def\CVG{\mathcal{C}_{V^G}}
\def\Ind{{\rm Ind}}
\def\Rep{{\rm Rep}}
\def\FPdim{{\rm FPdim}}
\def\MM{\mathcal{M}}
\def\VMM{\mathcal{VM}}
\def\NN{\mathcal{N}}
\def\HH{\mathcal{H}}
\def\h{\frak h}
\def\ha{\frac{1}{2}}
\def\bo{\boxtimes}
\def\u1{\underline{1}}
\newcommand{\Free}[1]{M(1)^{#1}}
\newcommand{\Fremo}[1]{{M(1,#1)}}
\newcommand{\charge}[1]{V_{L}^{#1}}
\newcommand{\charhalf}[2]{V_{#1+L}^{#2}}
\newcommand{\charlam}[1]{V_{#1+L}}
\newcommand\R[2]{{\bf R}_{#1,#2}}
\newcommand{\Fretw}[1]{M(1)(\tau)^{#1}}
\begin{center}
{\Large {\bf  Orbifolds and minimal modular extensions\footnote{This project is partially supported by NSFC grant 11871351, 12071314}}}\\
\vspace{0.5cm}

Chongying Dong
\\
 Department of Mathematics, University of
California, Santa Cruz, CA 95064 USA \\

Siu-Hung Ng\footnote{Partially supported by NSF grant
DMS1664418}\\
Department of Mathematics
Louisiana State University
Baton Rouge, LA 70803

Li Ren\\
 School of Mathematics,  Sichuan University,
Chengdu 610064 China
\end{center}

\begin{abstract} 
Let $V$ be a simple, rational, $C_2$-cofinite vertex operator algebra and $G$ a finite group acting faithfully on $V$ as automorphisms, which is simply called a rational vertex operator algebra with a $G$-action. It is shown that  the category $\EE_{V^G}$ generated by the $V^G$-submodules of $V$ is a symmetric fusion category braided equivalent to the $G$-module category $\EE=\Rep(G)$. If
$V$ is holomorphic, then the $V^G$-module category $\CC_{V^G}$ is a minimal modular extension of $\EE,$ and is equivalent to 
the Drinfeld center  $\ZZ(\Vec_G^{\alpha})$ as modular tensor categories  for some $\alpha\in H^3(G,S^1)$ with a canonical embedding of $\EE$.
Moreover, the collection $\MM_v(\EE)$ of equivalence classes of the minimal modular extensions $\CC_{V^G}$ of $\EE$ for  holomorphic vertex operator algebras $V$ with a $G$-action forms a group, which is isomorphic to a subgroup of $H^3(G,S^1).$ Furthermore, any  pointed modular category $\ZZ(\Vec_G^{\alpha})$ is equivalent to $\CC_{V_L^G}$
for some positive definite even unimodular lattice $L.$
In general, for any rational vertex operator algebra $U$ with a $G$-action,  $\CC_{U^G}$ is a minimal modular extension of the braided fusion subcategory $\FF$ generated by the $U^G$-submodules of $U$-modules. Furthermore, the group $\MM_v(\EE)$ acts freely on the set of equivalence classes $\MM_v(\FF)$ of the minimal modular extensions  $\CC_{W^G}$ of $\FF$ for any rational vertex operators algebra $W$ with a $G$-action. 
\end{abstract}

\section{Introduction}

This paper is a continuation of our study of $V^G$-module category $\CC_{V^G}$ for a regular (rational and $C_2$-cofinite) vertex operator algebra $V$ with a finite automorphism group isomorphic to $G$ of $V$ (cf. \cite{DNR}). It is established that if $V^G$ is regular, the category $\EE_{V^G}$ generated by the $V^G$-submodules of $V$ is a symmetric fusion category braided equivalent to the $G$-module category $\EE=\Rep(G).$ If $V$ is holomorphic, then the $V^G$-module category $\CC_{V^G}$ is a minimal modular extension of $\EE,$ and is equivalent to the Drinfeld center  $\ZZ(\Vec_G^{\alpha})$ or to the module category of twisted Drinfeld double $D^{\alpha}(G)$ for some $\alpha\in H^3(G,S^1)$ with a canonical embedding of $\EE$. This result has been conjectured  in \cite{DPR} where the $D^{\alpha}(G)$ was introduced and studied. Moreover, the collection $\MM_v(\EE)$ of equivalence classes of the minimal modular extensions $\CC_{V^G}$ of $\EE$ for some holomorphic vertex operator algebra $V$ with a $G$-action form a group, which is isomorphic to a subgroup of $H^3(G,S^1).$ Furthermore, any  pointed modular category $\ZZ(\Vec_G^{\alpha})$ is equivalent to $\CC_{V_L^G}$
for some positive definite even unimodular lattice $L.$ For any rational vertex operator algebra $U$ with a $G$-action,  $\CC_{U^G}$ is a minimal modular extension of the braided fusion subcategory $\FF$ generated by the $U^G$-submodules of $U$-modules. Furthermore, the group $\MM_v(\EE)$ acts freely on the set of equivalence classes $\MM_v(\FF)$ of the minimal modular extensions  $\CC_{W^G}$ of $\FF$ for any rational vertex operators algebra $W$ with a $G$-action.

It was proved in \cite{CM} that if $G$ is solvable, then the regularity of $V$ implies the regularity of $V^G.$ So we only need to assume $V$ is regular in this case. 
  More recently, the regularity of $V$ together with the $C_2$-cofiniteness of $V^G$ implies the rationality of $V^G$ \cite{Mc}.

We now give a detail discussion on this paper. 
 A braided fusion category $\CC$ over $\EE=\Rep(G)$, simply called a braided $\EE$-category, is a pair $(\CC, \eta)$ where $\CC$ is a braided fusion category and $\eta: \EE \to \CC$ is a full and faithful braided tensor functor. Throughout this paper, we call any full and faithful braided tensor functor  an  \emph{embedding}.  A braided $\EE$-category $(\CC, \eta)$ is said to be  \emph{nondegenerate} if $\eta: \EE \to \CC'$ is an equivalence, where $\CC'$ denotes the  M\"uger center of $\CC$. Note that $\EE$ is a nondegenerate braided $\EE$-category. We may simply write $\CC$ for the braided $\EE$-category $(\CC, \eta)$ when there is no ambiguity.  
 
An equivalence of  braided $\EE$-categories $(\CC_1, \eta_1)$ and $(\CC_2, \eta_2)$ is a  braided tensor equivalence $F: \CC_1 \to \CC_2$ such that  $\eta_2 \cong F \circ \eta_1$ as braided tensor functors.  
 
 A modular extension of a braided $\EE$-category $\CC$ is a pair $(\DD, j_\DD)$ in which $\DD$ is a modular tensor category and $j_\DD: \CC \to \DD$ is an embedding. Similar to the equivalence of two braided $\EE$-categories, two modular extensions $(\DD_1, j_1), (\DD_2, j_2)$ of $\CC$ said to be \emph{equivalent} if there exists a braided tensor equivalence $F: \DD_1 \to \DD_2$ such that $j_2\cong F \circ j_1$ as  braided tensor functors. We may simply write $\DD$ for the modular extension $(\DD, j_\DD)$ of $\CC$ if there is no ambiguity, and $[\DD]$ for the equivalence class of $(\DD, j_\DD)$. 
 
 A modular extension $\DD$ of a nondegenerate braided $\EE$-category $\CC$ is called \emph{minimal} if $\FPdim(\DD)=o(G)\cdot\FPdim(\CC).$
 According to \cite{BNRW} there are only finitely many inequivalent  minimal modular extensions of $\CC$ if there is one.  Moreover, by \cite{LKW1}, the collection $\MM(\EE)$ of equivalent classes $[\CC]$ of minimal modular extensions   of $\EE$ forms a finite group isomorphic to $H^3(G,S^1)$  under 
 the relative Deligne tensor product $\CC\otimes_\EE\DD$ for $[\CC],[\DD] \in \MM(\EE)$. 
 In fact, any minimal modular extension of $\EE$ is braided equivalent to the Drinfeld center $\ZZ(\Vec_G^{\alpha})$ where $\Vec_G^{\alpha}$ is the fusion category of $G$-graded vector spaces over $\C$ whose associativity isomorphism is given by the 3-cocycle $\alpha.$  Note that $\ZZ(\Vec_G^{\alpha})$ is  braided equivalent to the module category of the twisted Drinfeld double
 $D^{\alpha}(G).$ 
 
 For any pseudounitary nondegenerate braided $\EE$-category $\FF$, if $\FF$ has a minimal modular extension, then the collection $\MM(\FF)$ of equivalence classes of  minimal modular extensions of $\FF$ admits a natural action of $\MM(\EE)$ via the relative Deligne tensor product. Moreover, $\MM(\FF)$ is a $\MM(\EE)$-torsor \cite{LKW1}. 
 
Our investigation of  the $V^G$-module category $\CC_{V^G}$ in terms of the minimal modular extensions of certain braided fusion category is influenced greatly by the work of \cite{LKW1}. Note from  \cite{Hu} that $\CC_{V^G}$ is a modular tensor category. 
Associated to a rational vertex operator algebra $V$ with a $G$-action are two more braided fusion subcategories $\EE_{V^G}$ and  $\FF_{V^G}$ of $\CVG$. Here, $\EE_{V^G}$ is the full subcategory of $\CC_{V^G}$ generated by the $V^G$-submodules of $V$, and  $\FF_{V^G}$ is the full subcategory of $\CC_{V^G}$ generated by $V^G$-submodules of $V$-modules. $\FF_{V^G}$ is a nondegenerate braided $\EE$-category which satisfies
$$
C_{\CC_{V^G}}(\EE_{V^G}) = \FF_{V^G} \quad \text{and}\quad C_{\CC_{V^G}}(\FF_{V^G}) = \EE_{V^G}\,,
$$
where $C_\CC(\BB)$ denotes the M\"uger centralizer of the subcategory $\BB$ in the braided fusion category $\CC$. These two categories are the same if and only if $V$ is holomorphic. The main idea is to put these categories in the context of the minimal modular extensions. Recall from \cite{DLM1} a Schur-Weyl type duality decomposition
$$V=\oplus_{\lambda\in \irr(G)}W_{\lambda}\otimes V_{\lambda}$$
 where $W_{\lambda}$ is the irreducible $G$-module with character $\lambda$ and the multiplicity spaces $V_{\lambda}$ are inequivalent irreducible $V^G$-modules. Then $\EE_{V^G}$ is generated by $V_{\lambda}$ for $\lambda\in \Irr(G).$ Our first 
 result asserts that $\EE_{V^G}$ is a symmetric fusion category braided equivalent to $\EE$ for any rational vertex operator algebra $V$ via an embedding  $F^{V,G}: \EE \to \CC_{V^G}$
 (also see \cite{Ki}). In particular, $(\CC_{V^G}, F^{V,G})$ is a braided $\EE$-category.
 
 In the case when $V$ is holomorphic, $(\CC_{V^G}, F^{V,G})$ is a minimal modular extension of $\EE.$ So $\CC_{V^G}$ is braided 
 equivalent to $\ZZ(\Vec_G^{\alpha})$  for some $\alpha\in H^3(G,S^1).$ Let ${\bf H}_G$ be the collection  holomorphic vertex operator algebras with a  $G$-action. Then $\MM_v(\EE)$ consisting of equivalence classes of $(\CC_{V^G}, F^{V,G})$ for some $V\in {\bf H}_G$ is a subgroup of $\MM(\EE).$ We certainly believe that $\MM_v(\EE)=\MM(\EE)$. If $G$ is an abelian group generated by less than 3 elements or an odd dihedral group,
 we show that  $\MM_v(\EE)=\MM(\EE).$
 We also show that the group operation given in \cite{LKW1} can be realized from the tensor product of vertex operator algebras, i.e.
 $\CC_{V^G}\ot^{F^{V,G}, F^{U,G}}_{\EE} \CC_{U^G} \cong (\CC_{(V\otimes U)^G}, F^{V \ot U, G})$ for $V, U\in {\bf H}_G,$ where the $G$-action on $V \ot W$ is the diagonal action of $G \times G$.

 Now we assume that $\FF$ is an  arbitrary pseudounitary nondegenerate braided $\EE$-category. Let ${\bf R}_{G}^{\FF}$ be the collection of rational vertex operator algebras $W$ with a  $G$-action  such that $\FF_{W^G}$ is  equivalent to $\FF$ as braided  $\EE$-categories. If ${\bf R}_G^\FF$ is not empty, we establish that $\MM_v(\EE)$  acts freely on the set of equivalence classes $\MM_v(\FF)=\{[\CC_{W^G}]|W\in {\bf R}_G^{\FF} \text{ and }\}$
  such that  $\CC_{V^G}\ot^{F^{V,G}, F^{W,G}}_{\EE} \CC_{W^G} \cong (\CC_{(V\otimes W)^G}, F^{V \ot W, G})$ for $V\in {\bf H}_G,$ $W\in {\bf R}_{G}^{\FF}$ where the $G$-action of $V \ot W$ is the diagonal action of $G \times G$. Again, it is desirable that $\MM_v(\FF)=\MM(\FF)$  and $\MM_v(\FF)$ is a $\MM_v(\EE)$-torsor  whenever $\MM_v(\FF) \ne \emptyset$.
 
For any braided fusion category $\CC$ with braiding isomorphism $c_{X,Y}:X\otimes Y\to Y\otimes X$, $\overline{c}_{X,Y}=c_{Y,X}^{-1}$ also defines a braiding on $\CC$. We denote by $\ol{\CC}$ the braided fusion category $\CC$ equipped with the braiding $\ol{c}$. Note that  $\ol{\EE} = \EE$ as braided tensor categories. Thus, if $(\CC, j)$ is a braided  $\EE$-category, then so is $(\ol \CC, j)$. Again, we will simply write $\ol\CC$ for the braided $\EE$-category $(\ol\CC, j)$.  The  braided $\EE$-category $\overline{\CC}$ plays an essential role in the group structure of $\MM(\EE)$
and the $\MM(\EE)$-torsor structure on $\MM(\FF).$ In fact, if $\MM\in\MM(\EE)$, then its inverse is exactly 
 $\overline{\MM}.$ The proof of free and transitive action of  $\MM(\EE)$ on  $\MM(\FF)$ in \cite{LKW2} also uses $\overline{\NN}$ for 
$\NN\in \MM(\FF).$  So it is necessary and important to understand $\overline{\CC_V}$ for a rational vertex operator algebra $V.$ For such $V$ we now have two modular tensor categories $\CC_V$ and $\overline{\CC_V}.$ In the setting of braiding isomorphism, one needs to define $(-1)^n$ for rational number $n.$ The braidings $c_{X,Y}$ and $c_{Y,X}^{-1}$ correspond to the choices $(-1)^n=e^{\pi in}$ and $(-1)^n=e^{-\pi in}.$ 

From the point of view of vertex operator algebra, we conjecture that the modular tensor category $\overline{\CC_V}$ is braided equivalent to $\CC_U$ for some rational vertex operator algebra $U.$ This is consistent with the reconstruction program. That is,  any modular tensor category $\CC$ can be realized as $\CC_{W}$ for some rational vertex operator algebra. Using the language of vertex operator algebra, the conjecture is equivalent to the statement: If $V$ is a rational vertex operator algebra then there exists a holomorphic vertex operator algebra $H$ containing $V$ as sub-VOA  such that the double commutant $C_H(C_H(V))=V.$ Then $\overline{\CC_V}$ and
$\CC_{C_H(V)}$ are braided equivalent. We prove this conjecture for lattice vertex operator algebra $V_L,$ affine vertex operator algebras associated to the integrable highest weight representations and the Virasoro vertex operator algebras associated to the discrete series. 

This paper is organized as follows: We review the twisted modules and $g$-rationality of  vertex operator algebras following \cite{DLM3} in Section 2. Section 3 is a review of basics on the fusion categories, braided fusion categories, modular tensor categories and minimal modular extensions of a fusion category over $\EE$ \cite{ENO,EGNO,KO}.  We also present the main results concerning $\MM(\EE)$ and its torsor $\MM(\CC)$ from \cite{LKW1}. Section 4 is a review of the modular tensor category 
$\CC_V$ associated to a rational, $C_2$-cofinite vertex operator algebra $V$ \cite{HL1,HL2,HL3, Hu}. We 
discuss how to realize $\overline{\CC_V}$  as $\CC_U$ with some conjecture and examples in the last half of this section. We also
gives a necessary and sufficient condition for $\overline{\CC_V}$ and $\CC_U$ being braided equivalent. In Section 5, we recall from \cite{DRX, DLXY} the classification of irreducible $V^G$-modules and related results. In Section 6 we  prove that for any
rational, $C_2$-cofinite vertex operator algebra $V,$ $\EE$ and $\EE_{V^G}$ are braided equivalent, and 
the regular commutative algebra $\C[G]^*$ in $\EE$ corresponds to commutative algebra $V$ in $\EE_{V^G}$ for any rational vertex operator algebra $V$ under the braided equivalence. Furthermore, $\CC_{V^G}$ is a minimal modular extension of $\FF_{V^G}.$ In particular, if $V$ is holomorphic, then $\CC_{V^G}$ is a minimal modular extension of $\EE.$ Section 7 is devoted to the proof that $\MM_v(\EE)$ is a finite abelian group under the product $\CC_{V^G}\cdot \CC_{U^G}=\CC_{(V\otimes U)^G}$ and $\MM_v(\EE)$ acts on $\MM_v(\FF)$ by $\CC_{V^G}\cdot \CC_{W^G}=\CC_{(V\otimes W)^G}.$ We prove in Section 8 that if  $\ZZ(\Vec_G^{\alpha})$ is pointed, then $\ZZ(\Vec_G^{\alpha})$ is equivalent to $\CC_{V_L^G}$
for some positive definite even unimodular lattice $L.$

\section{Twisted modules}
Let $V$ be a vertex operator algebra and $g$ an automorphism of $V$ of finite order $T$. Then  $V$ is a direct sum of eigenspaces of $g:$
$V=\bigoplus_{r\in \Z/T\Z}V^r$
where $V^r=\{v\in V|gv=e^{-2\pi ir/T}v\}$.
We use $r$ to denote both
an integer between $0$ and $T-1$ and its residue class \m $T$ in this
situation.

A {\em weak $g$-twisted $V$-module} $M$ is a vector space equipped
with a linear map
\begin{equation*}
\begin{split}
Y_M: V&\to (\End\,M)[[z^{1/T},z^{-1/T}]]\\
v&\mapsto\displaystyle{ Y_M(v,z)=\sum_{n\in\frac{1}{T}\Z}v_nz^{-n-1}\ \ \ (v_n\in
\End\,M)},
\end{split}
\end{equation*}
which satisfies the following:  for all $0\leq r\leq T-1,$ $u\in V^r$, $v\in V,$
$w\in M$,
\begin{eqnarray*}
& &Y_M(u,z)=\sum_{n\in \frac{r}{T}+\Z}u_nz^{-n-1} \label{1/2},\\
& &u_lw=0~~~
\mbox{for}~~~ l\gg 0,\label{vlw0}\\
& &Y_M({\mathbf 1},z)=\Id_M,\label{vacuum}
\end{eqnarray*}
 \begin{equation*}\label{jacobi}
\begin{array}{c}
\displaystyle{z^{-1}_0\delta\left(\frac{z_1-z_2}{z_0}\right)
Y_M(u,z_1)Y_M(v,z_2)-z^{-1}_0\delta\left(\frac{z_2-z_1}{-z_0}\right)
Y_M(v,z_2)Y_M(u,z_1)}\\
\displaystyle{=z_2^{-1}\left(\frac{z_1-z_0}{z_2}\right)^{-r/T}
\delta\left(\frac{z_1-z_0}{z_2}\right)
Y_M(Y(u,z_0)v,z_2)},
\end{array}
\end{equation*}
where $\delta(z)=\sum_{n\in\Z}z^n$ and
all binomial expressions (here and below) are to be expanded in nonnegative
integral powers of the second variable.

A $g$-{\em twisted $V$-module} is
a $\C$-graded weak $g$-twisted $V$-module $M:$
\begin{equation*}
M=\bigoplus_{\lambda \in{\C}}M_{\lambda}
\end{equation*}
where $M_{\l}=\{w\in M|L(0)w=\l w\}$ and $L(0)$ is the component operator of $Y(\omega,z)=\sum_{n\in \Z}L(n)z^{-n-2}.$ We also require that
$\dim M_{\l}$ is finite and for fixed $\l,$ $M_{\frac{n}{T}+\l}=0$
for all small enough integers $n.$ If $w\in M_{\l}$, we refer to $\l$ as the {\em weight} of
$w$ and write $\l=\wt w.$

We use $\Z_+$ to denote the set of nonnegative integers.
 An {\em admissible} $g$-twisted $V$-module
is a  $\frac1T{\Z}_{+}$-graded weak $g$-twisted $V$-module $M:$
\begin{equation*}
M=\bigoplus_{n\in\frac{1}{T}\Z_+}M(n)
\end{equation*}
satisfying
\begin{equation*}
v_mM(n)\subseteq M(n+\wt v-m-1)
\end{equation*}
for homogeneous $v\in V,$ $m,n\in \frac{1}{T}{\Z}.$

If $g=\Id_V$,  we have the notions of  weak, ordinary and admissible $V$-modules \cite{DLM3}.

If $M=\bigoplus_{n\in \frac{1}{T}\Z_+}M(n)$
is an admissible $g$-twisted $V$-module, the contragredient module $M'$
is defined as follows:
\begin{equation*}
M'=\bigoplus_{n\in \frac{1}{T}\Z_+}M(n)^{*},
\end{equation*}
where $M(n)^*=\Hom_{\C}(M(n),\C).$ The vertex operator
$Y_{M'}(a,z)$ is defined for $a\in V$ via
\begin{eqnarray*}
\langle Y_{M'}(a,z)f,w\rangle= \langle f,Y_M(e^{z L(1)}(-z^{-2})^{L(0)}a,z^{-1})w\rangle,
\end{eqnarray*}
where $\langle f,w\rangle=f(w)$ is the natural paring $M'\times M\to \C.$
It follows from \cite{FHL} and \cite{X} that $(M',Y_{M'})$ is an admissible $g^{-1}$-twisted $V$-module. The  $g^{-1}$-twisted $V$-module $M'=(M',Y_{M'})$ is called the contragredient module of the $g$-twisted $V$-module $M.$ Moreover, $M$ is irreducible if and only if $M'$ is irreducible.

A \voa $V$ is called $g$-rational, if the  admissible $g$-twisted module category is semisimple. $V$ is called rational if $V$ is $1$-rational.
A \voa $V$ is $C_2$-cofinite if $V/C_2(V)$ is finite dimensional, where $C_2(V)=\langle v_{-2}u|v,u\in V\rangle$ \cite{Z}. A \voa $V$ is called regular if every weak $V$-module is a direct sum of irreducible $V$-modules \cite{DLM2}. It is proved  in \cite{ABD} that if  $V$ is of CFT type, then regularity is equivalent to rationality and $C_2$-cofiniteness. Also $V$ is regular if and only if the weak module category is semisimple \cite{DYu}.

The following results about $g$-rational \voas \  are well-known \cite{DLM3}, \cite{DLM4}.
\begin{thm}\label{grational}
If $V$ is $g$-rational, then:
\begin{enumerate}
    \item[\rm (1)] Any irreducible admissible $g$-twisted $V$-module $M$ is a $g$-twisted $V$-module. Moreover, there exists a number $\l \in \mathbb{C}$ such that  $M=\oplus_{n\in \frac{1}{T}\mathbb{Z_+}}M_{\l +n}$ where $M_{\lambda}\neq 0.$ The $\l$ is called the conformal weight of $M;$

 \item[\rm (2)] There are only finitely many irreducible admissible  $g$-twisted $V$-modules up to isomorphism.

 \item[\rm (3)] If $V$ is also $C_2$-cofinite and $g^i$-rational for all $i\geq 0$ then the central charge $c$ and the conformal weight $\l$ of any irreducible $g$-twisted $V$-module $M$ are rational numbers.
 \end{enumerate}
\end{thm}

A \voa $V=\oplus_{n\in \Z}V_n$  is said to be of CFT type if $V_n=0$ for negative $n$ and $V_0=\C {\bf 1}.$

\section{Fusion categories}

In this section we will review  fusion categories and modular tensor categories following \cite{ENO}, \cite{EGNO}, \cite{KO}. A  fusion category ${\cal C}$ is a $\C$-linear abelian semisimple, rigid monoidal category with finitely many inequivalent simple objects, and finite dimensional morphism spaces, together with a tensor product functor $\boxtimes: {\cal C}\times {\cal C}\to {\cal C},$ a unit object $\1_\CC$ satisfying certain axioms. We use $X'$ to denote the (left) dual object of $X\in\CC$, and 
$\Or(\Cc)$ denotes the set of equivalence classes of the simple objects.   Throughout this paper, subcategories  of $\Cc$ are always assumed to be full.  A fusion subcategory of $\CC$ is defined as expected.

A very useful concept is  so called the \emph{Frobenius-Perron dimension}. Let $K_0(\CC)$ be the Grothendieck ring of a fusion category $\CC.$ Then there is a unique ring homomorphism $\FPdim: K_0(\CC)\to \BR$ satisfying
$\FPdim(M)\geq 1$ for any nonzero object $M.$ The Frobenius-Perron dimension of $\CC$ is defined to be $\FPdim(\CC)=\sum_{M\in \Or(\CC)}\FPdim (M)^2.$ In the case   $\CC$ is a fusion subcategory of the module category for
a vertex operator algebra $V$, the Frobenius-Perron dimension $\FPdim (M)$ is exactly the quantum dimension $\qdim_V(M)$ studied in \cite{DJX} and \cite{DRX}.

 A \emph{braided fusion category} is a fusion category $\CC$  with a natural isomorphism $c_{X,Y} : X \boxtimes Y\to  Y\boxtimes X$, called a \emph{braiding}, which satisfies some compatible conditions. Associated to a braided fusion category $\CC$ is another braided fusion category $\overline{\CC}$ which has the same fusion category as $\CC$ with a new braiding $\overline{c}_{X,Y}=c_{Y,X}^{-1}.$ A braided fusion category $\CC$  is called \emph{symmetric} if  $c_{Y,X}\circ c_{X,Y} =\id_{X\boxtimes Y}$
 or $\CC=\overline{\CC}$ as braided fusion categories. For any collection $\DD$ of objects in $\CC$,  the \emph{M\"uger centralizer} $C_{\CC}(\DD)$ is the  subcategory of $\CC$ consisting of the objects $Y$ in $\CC$ such that $c_{Y,X}\circ  c_{X,Y} = \id_{X\boxtimes Y}$  for all $X$ in $\DD.$ The subcategory $C_\CC(\DD)$ is closed under the tensor product of $\CC$ and hence a braided fusion subcategory of $\CC$. The symmetric fusion category $C_{\CC}(\CC)$ is called the \emph{M\"uger center} of $\CC$, and denoted by $\CC'$. For example, for any finite group $G,$ the finite dimensional $\C [G]$-module
 category ${\rm Rep}(G)$ is a symmetric fusion category with the usual tensor product and braiding of $\C$-linear spaces. A symmetric fusion category $\CC$
is called \emph{Tannakian} if there is a finite group $G$ such that $\CC$ is equivalent to 
${\rm Rep}(G)$  as braided fusion categories.  According to \cite{De},  the braided fusion category $\CC$ is Tannakian if and only if there  exists a faithful braided tensor functor from
$\CC$ to $\Vec$, where $\Vec$ denotes the category of finite dimensional $\C$-linear spaces with the usual tensor product and braiding.

A braided fusion category  $\CC$ is called \emph{nondegenerate} if $\CC' \stackrel{\otimes}{\cong} \Vec$. A nondegenerate spherical braided fusion category is called a modular tensor category. This definition is equivalent to that the corresponding $S$-matrix is nonsingular \cite{Mu2}. From \cite{DGNO}  we know that if $\CC$ is a braided fusion category and $\BB$ is a fusion subcategory then
\begin{equation}\label{3.1}
\FPdim(\BB)\cdot \FPdim (C_\CC(\BB))=\FPdim(\CC)\cdot \FPdim(\CC'\cap\BB).
\end{equation}

The Drinfeld center $\ZZ(\CC)$ of a fusion category $\CC$ is a braided fusion category whose objects are pairs $(X, z_{X,-})$ in which $X\in\CC$ and  $z_{X,-}: X\boxtimes (-)\to (-)\boxtimes X$ a natural isomorphism, called \emph{a half-braiding}, satisfying certain conditions. Moreover, $\FPdim(\ZZ(\CC))=\FPdim(\CC)^2.$ If $\CC$ is  a spherical fusion category then $\ZZ(\CC)$ is a modular tensor category \cite{Mu2}. In particular, $\ZZ(\Rep(G))$ is a modular tensor category.

Let $\CC$ be a braided fusion category. Then $\EE=\CC'$ is a symmetric fusion category.  A  modular extension (ME) of $\CC$ is a pair $(\DD, \i_\DD)$ where $\DD$ is modular tensor category and $\i_\DD: \CC \to \DD$ is a full and faithful braided tensor  (and simply called an  \emph{embedding} in the sequel). The ME $(\DD, j_\DD)$ is called \emph{minimal} (MME) if $C_{\DD}(\EE)=\CC$ under the identification of $\i_\DD(\CC)$ with $\CC$. We will simply write $\DD$ for an ME $(\DD, \i_\DD)$ of $\CC$ and identify $\CC$ with $\i_\DD(\CC)$ when the context is clear.

\begin{lem}\label{l3.1} A modular extension $\DD$ of a braided fusion category $\CC$ over $\EE$ is minimal if and only if  
$$\FPdim(\DD)=\FPdim(\CC)\cdot \FPdim(\EE).$$
\end{lem}
\begin{proof} Since the modular tensor category $\DD$ is nondegenerate, we see that $\DD'=\Vec$ and $\FPdim(\DD'\cap \EE)=1.$
It follows immediately from equation (\ref{3.1})  
 since $\CC$ is always a fusion subcategory of $C_{\DD}(\EE).$
\end{proof}

Let $\CC$ be any braided fusion category over $\EE$. Two  modular extensions $(\DD_1,\iota_1)$ and $(\DD_2,\iota_{2})$ of $\CC$ are equivalent if there is a braided equivalence $F: \DD_1\to \DD_2$ such that $F \circ \i_{1} \cong \i_2$ as braided tensor functors. Let $\MM(\CC)$ be the set of equivalence classes of MMEs of $\CC.$ Then  $\MM(\CC)$ is a finite set as every MME of $\CC$ has the same
 Frobenius-Perron dimension $\FPdim(\CC)\cdot \FPdim(\EE)$ and there are only finitely many modular tensor categories up to equivalence for any fixed
 Frobenius-Perron dimension \cite{BNRW}. The following important result was obtained in \cite{LKW1}.
 \begin{thm}\label{LKW1}
 Let $\CC$ be a braided $\EE$-category. Then $\MM(\EE)$ is a finite abelian group and  $\MM(\EE)$ acts on  $\MM(\CC)$  freely and transitively provided $\MM(\CC) \ne \emptyset$.
 In particular, the cardinality of $\MM(\CC)$ equals to the order of  $\MM(\EE)$ if  $\MM(\CC) \ne \emptyset$.
  \end{thm}

 The definition of the product on $\MM(\EE)$  and the action of  $\MM(\EE)$ on  $\MM(\CC)$  are quite complicated, and we will discuss later in Section 6 in details.
 
 An object $A$ in a braided fusion category $\CC$ is called a \emph{commutative algebra}  if there are morphisms
$\mu: A\boxtimes A\to A$ and $\eta: {\bf 1_\CC}\to A$ such that $$\mu\circ(\mu\boxtimes \id_A)\circ \alpha_{A,A,A} = \mu\circ(\id_A\boxtimes\mu),\ \mu = \mu\circ c_{A,A}$$
$$\mu\circ(\eta\boxtimes \id_A)\circ l_A^{-1}=\id_A =\mu\circ(\id_A\boxtimes\eta) \circ r_A^{-1}$$
 where $\alpha_{A,A,A}: A\boxtimes (A\boxtimes A) \to (A\boxtimes A)\boxtimes A$ is the associativity isomorphism, and  $l_A: \1_\CC\boxtimes A\to A$ and $r_A:A \bo \1_\CC \to A$ are respectively the left and the right unit isomorphisms.  A commutative algebra $A$ in $\CC$ is called \emph{connected} if $\dim \Hom_\CC(\1_\CC, A)=1$.

Let $A$ be a connected commutative algebra in $\CC$. A right $A$-module $M$ is an object in $\CC$ with a morphism $\mu_M:  M \bo A \to M$ such that $\mu_M\circ ( \id_M \bo \mu)=\mu_M\circ ( \mu_M \bo \id_A)\circ \alpha_{M, A,A}$ and 
$\mu_M\circ(\id_M \bo \eta)=r_M$ where $r_M:M \bo \1_\CC \to M$ is the right unit isomorphism.  If $M, N$ are right $A$-modules, a morphism $f: M \to N$ in $\CC$ is called an $A$-module morphism if $\mu_N \circ(f \bo \id_A) = \mu_M$. We denote the category of right $A$-modules by $\CC_A.$ 

The left $A$-modules are defined similarly. Since $A$ is commutative, a right $A$-module $M$ admits two natural left $A$-module structure on $M$, namely $m, \ol m : A \bo M \to M$ defined by
$$
m= \mu_M \circ c_{A,M} \quad \text{and}\quad \ol m = \mu_M \circ \ol c_{A,M}\,.
$$
These left $A$-module structures on $M$ defines two $A$-bimodule structures on $M$, and they coincide when $M \in C_\CC(A)$. We will consider any right $A$-module as an $A$-bimodule under the left $A$-action $\ol m$ as discussed. An $A$-module $M$ is called \emph{local} if
$\mu_M\circ c_{M,A}\circ c_{A,M}=\mu_M.$ We denote the local $A$-module category by $\CC_A^0.$ It is immediate to see that the $A$-modules in $C_\CC(A)$ are local $A$-modules.

An algebra $A$ in $\CC$ is an $A$-bimodule under product map $\mu$ of an algebra $A$. If $\mu: A \boxtimes A \to A$ splits as $A$-bimodule morphism in $\CC$, then $A$ is called \emph{separable}. Following the terminology in \cite{LKW1}, an algebra $A$ in ribbon tensor category $\CC$ is said to be \emph{condensable} if $A$ is a commutative, separable and connected algebra in $\CC$ with $\dim(A)\ne 0$ and $\theta_A=\id_A$.

The $A$-module category,  $\CC_A$, is a fusion category with the tensor product $M\boxtimes_A N$, where $N$ is considered as an $A$-bimodule under the preceding convention. Moreover, the category $\CC_A^0$ of local $A$-modules is a braided fusion category. Moreover, if $\CC$ is modular tensor category and $A$ is a condensable algebra in $\CC$,  then $\CC_A^0$ is modular \cite{KO}. We will only consider pseudounitary fusion category $\CC$ with $\dim(X) = \FPdim(X)$ for all object $X \in \CC$. In the case, the condition $\dim(A) \ne 0$ is satisfied automatically. Moreover, an embedding of (pseudounitary) braided fusion categories preserves the canonical pivotal structures and hence their ribbon  structures. 

The following identities give relations among dimensions of relevant categories $\CC$ and condensable algebras $A$  in $\CC$:
$$\FPdim(\CC_A) =\frac{\FPdim(\CC)}{\FPdim(A)},\  \FPdim(\CC_A^0) =\frac{\FPdim(\CC)}{\FPdim(A)^2}$$
where the first identity holds for any braided fusion category $\CC$ and the second identity requires that $\CC$ is modular \cite{DMNO}.

We have mentioned that for a  finite group $G,$ $\Rep(G)$ is a symmetric fusion category. Then $A=\C[G]^*$ is a condensable in $\Rep(G)$, called the \emph{regular algebra} of $\Rep(G)$. Then,  $\Rep(G)_A=\Rep(G)_A^0$ is equivalent to the category $\Vec$ of finite dimensional vector spaces.  Furthermore, any 
condensable algebra in $\Rep(G)$ is given by $\C[G/H]^*$ where $H$ is a subgroup of $G$ \cite{KO}.

We now discuss the modular extension of $\EE=\Rep(G).$ Let $(\MM,\iota_\MM)$ be an MME  of $\EE.$ Then  the regular algebra $A$ of $\EE$ is a condensable algebra in  $\MM$. Following \cite{DGNO},  $\MM_A$ is a pointed fusion category equivalent to $\Vec_G^\alpha$ for some $\alpha\in H^3(G,S^1).$
That is, $\MM_A=\bigoplus_{g\in G}(\MM_A)_g$ and each $(\MM_A)_g\cong \Vec$ as $\C$-linear categories and 
$$(\MM_A)_g\boxtimes_A (\MM_A)_h\cong (\MM_A)_{gh}.$$ We denote the simple 
object of $(\MM_A)_g$ by $e(g)$ up to isomorphism. Then 
$$\alpha: (e(g)\boxtimes_A( e(h))\boxtimes_A e(k) \to e(g)\boxtimes_A (e(h) \boxtimes_A e(k))$$
gives the associativity isomorphism of $\MM_A$. Moreover, $\MM$ and $\ZZ(\Vec_G^\alpha)$ are braided equivalent. Furthermore, $\ZZ(\Vec_G^\alpha)$ is an MME of $\Rep(G)$ for any $\alpha\in H^3(G,S^1).$ It is well known that $\ZZ(\Vec_G^\alpha)$ is braided equivalent to the representation category of the twisted Drinfeld double $D^{\alpha}(G)$ of $G$ \cite{DPR}. It was proved in \cite{LKW1} that $(\MM, \iota_\MM)$ is equivalent to $(\ZZ(\Vec_G^\alpha), \iota_\alpha)$ where $\iota_\alpha: \EE \to \ZZ(\Vec_G^\alpha)$ is the canonical embedding, which can be described as follows: 
Recall that the center $\ZZ(\Vec_G^\a)$ of $\Vec_G^\a$ consists of the pairs $(X, c_{X, -})$, in which $X \in \Vec_G^\a$ and $c_{X, Y}: X \ot Y\to Y \ot X$, called an \emph{half-braiding}, is a natural isomorphism for $Y \in \Vec_G^\a$ satisfying compatibility conditions (cf. \cite{EGNO} for the center construction). For each $X \in \Rep(G)$, we consider $X$ as a homogeneous vector space of grading 1, and we define the half-braiding $c_{X,-}$ by setting $c_{X, e(g)} : X \ot e(g) \to e(g) \ot X$, $x \ot 1 \mapsto 1 \ot g^{-1} x$ for $x \in X$. This assignment $X \mapsto (X, c_{X, -})$ can be extended to an embedding, i.e., a faithful and full braided tensor functor, $\iota_\a: \EE \to \ZZ(\Vec_G^\a)$. 

By \cite{LKW1}, $(\MM,\iota_\MM) \cong (\ZZ(\Vec^\a_G), \iota_\a)$ as braided $\EE$-categories, and  the map $\Phi_G: \a \mapsto (\ZZ(\Vec^\a_G), \iota_\a)$ defines a group isomorphism from $H^3(G, S^1)$ to $\MM(\EE)$.
\section{Modular categories associated to vertex operator algebras}

Let $V$ be a rational, $C_2$-cofinite vertex operator algebra of CFT type such that the conformal weight of any irreducible $V$-module $M$  is nonnegative, and is zero if and only if $M=V.$ Then the $V$-module category $\CC_V$ is a modular tensor category
\cite{Hu}. For the purpose of later discussion we need  details on the tensor product of two modules explicitly. So we first give a brief review on the construction of the tensor product  of modules for $V$-modules  from \cite{HL1,HL2,HL3,Hu}. 

For any complex number  $z\in\C^{\times}=\C\setminus \{0\},$
there is a tensor product $W_1\boxtimes_{P(z)} W_2$ for any $V$-modules $W_1,W_2$ together with a canonical intertwining operator $I=I_{\boxtimes_{P(z)}}$ of type  $\left(_{W_1,\ \ W_2}^{W_1\boxtimes_{P(z)} W_2}\right)$ satisfying some universal property such that for $w_i\in W_i,$ there is a tensor element 
$${w_1\boxtimes_{P(z)}w_2}=I(w_1,z)w_2\in\overline{W_1\boxtimes_{P(z)} W_2}$$
where $\overline{W_1\boxtimes_{P(z)} W_2}$ is the formal completion of $W_1\boxtimes_{P(z)} W_2$. The operator $I(w_1,z)$ is  understood to be $\sum_{n \in \BR} (w_1)_n z^{-n-1}$ where $(w_1)_n \in \Hom(W_2, W_1\boxtimes_{P(z)} W_2)$, and $z^n=e^{n\log z}$ for any $n\in\BR$ with $\log z=\log |z|+i {\rm arg}z$ and $0\leq{\rm arg}z<2\pi.$ 
Moreover, $W_1\boxtimes_{P(z)} W_2$ is spanned by the coefficients
of $z^n$ for all $w_i$ and $n\in \BR.$ If we have three $V$-modules $W_i$ for $i=1,2,3$ there is an associativity isomorphism
$$ A_{z_1,z_2}: W_1\boxtimes_{P(z_1)}(W_2\boxtimes_{P(z_2)}W_3) \to (W_1\boxtimes_{P(z_1-z_2)}W_2)\boxtimes_{P(z_2)}W_3$$
characterized by 
$$w_1\boxtimes_{P(z_1)}(w_2\boxtimes_{P(z_2)}w_3) \mapsto (w_1\boxtimes_{P(z_1-z_2)}w_2)\boxtimes_{P(z_2)}w_3$$
for $|z_1|>|z_2|>|z_1-z_2|>0.$ The tensor product in the modular tensor category $\CC_V$ is given by $\boxtimes=\boxtimes_{P(1)}.$ We will simply denote the corresponding intertwining operator $I_{\boxtimes_{P(1)}}$ by $I.$

To discuss the braiding and associativity isomorphism in $\CC_V$ we also need the natural parallel transport isomorphisms.
 Fix $V$-modules $W_1,W_2$ and nonzero complex numbers $z_1,z_2$, for any continuous path $\gamma$ in $\C^\times$  from $z_1$ to $z_2,$  the parallel transport isomorphisms  
  $T_{\gamma}:W_1\boxtimes_{P(z_1)} W_2\to W_1\boxtimes_{P(z_2)} W_2 $ is determined by the extension
  \begin{eqnarray*} \overline{T_{\gamma}  }: &\overline{W_1\boxtimes_{P(z_1)} W_2} &\to   \overline{W_1\boxtimes_{P(z_2)} W_2}  \\
  &w_1\boxtimes_{P(z_1)} w_2 &\mapsto  I_{\boxtimes_{P(z_2)}}(w_1,e^{l(z_1)})w_2  
  \end{eqnarray*}
  where $l(z_1)$ is the value of $\log z_1$ determined uniquely by $\log z_2$ with ${\rm arg}z_2\in [0,2\pi)$ and the path. 
  The braiding isomorphism $c_{W_1,W_2}: W_1\boxtimes W_2\to W_2\boxtimes W_1$ is determined by 
  $$\overline{c_{W_1,W_2}}(w_1\boxtimes w_2)=e^{L(-1)}\overline{T_{\gamma^-}}(w_2\boxtimes_{P(-1)}w_1)=e^{L(-1)}I(w_2, e^{\pi i})w_1$$
  where $\gamma^-$ is a path on the upper half plane without $0$ from $-1$ to $1.$ Then $c_{W_2,W_1}^{-1}$ is determined 
  by 
  $$\overline{c_{W_2,W_1}^{-1}}(w_1\boxtimes w_2)=e^{L(-1)}\overline{T_{\gamma^+}}(w_2\boxtimes_{P(-1)}w_1)=e^{L(-1)}I(w_2, e^{-\pi i})w_1$$
  $\gamma^+$ is a path on the lower half plane without $0$ from $-1$ to $1.$ 
  The associativity isomorphism 
  $$A_{W_1,W_2,W_3}: W_1\boxtimes(W_2\boxtimes W_3) \to (W_1\boxtimes W_2)\boxtimes W_3$$
  is given by
 $$ A_{W_1,W_2,W_3}=  T_{\gamma_3} \circ (T_{\gamma_4} \boxtimes_{P(z_2)}\id_{W_3} ) \circ A_{z_1,z_2} \circ (\id_{W_1}\boxtimes_{ P(z_1) }T_{\gamma_2} )\circ T_{\gamma_1}$$
 where $z_1>z_2>z_1-z_2>0,$ $\gamma_1,\gamma_2$ are paths in $\BR^\times=\BR\setminus\{0\}$  from $1$ to $z_1,$ $z_2,$ respectively,  and $\gamma_3,\gamma_4$ are paths in the real line with $0$ from $z_2,$ $z_1-z_2$ to $1,$ respectively. 

We now investigate more on the modular tensor category $\overline{\CC_V}$ with braiding $c_{W_2,W_1}^{-1}.$
The difference between these two braidings is how we choose $(-1)^n$ for any rational number $n.$ For braiding $c_{W_1,W_2},$ 
$(-1)^n$ is understood to be $e^{\pi in}$ and  for braiding $c_{W_2,W_1}^{-1},$ 
$(-1)^n$ is understood to be $e^{-\pi in}$ (see the proof of Theorem \ref{t5.2}). So the two different braidings in the $V$-module category really comes from the two different ways of choosing the skew symmetry 
for intertwining operators which define the tensor product. 

 There is a twist $\theta_W=e^{2\pi i L(0)}: W\to W$ for any $W\in \CC_V.$ If $W$ is irreducible then $\theta_W=e^{2\pi i \Delta_W}$ where 
$\Delta_W$ is the weight of $W.$ Then the twist in the modular tensor category $\overline{\CC_V}$ is given by $\overline{\theta}_W=e^{-2\pi iL(0)}.$
If $W$ is irreducible,  $\overline{\theta}_W$ is exactly the complex conjugation of $\theta_W$ as $\Delta_W$ is a rational number. The relation
$\ol{\theta}_{W_1\boxtimes W_2}=\ol{c}_{W_2,W_1}\circ \ol{c}_{W_1,W_2}\circ (\ol{\theta}_{W_1}\boxtimes \ol{\theta}_{W_2})$
is immediate by taking the inverse of the relation $\theta_{W_1\boxtimes W_2}=c_{W_2,W_1}\circ c_{W_1,W_2}\circ (\theta_{W_1}\boxtimes \theta_{W_2}).$

Here is a natural question: Assume that $V$ is a rational, $C_2$-cofinite vertex operator algebra. Is there a rational vertex operator algebra $U$ such that $\overline{\CC_V}$ and $\CC_U$ are braided equivalent? 

 \begin{conj}\label{conjecture4.1} Assume that $V$ is a rational, $C_2$-cofinite vertex operator algebra. Then there is a rational, $C_2$-cofinite vertex operator algebra $U$ such that $\overline{\CC_V}$ and $\CC_U$ are braided equivalent. 
 \end{conj}
 
 Let  $W=(W,Y,\1,\omega)$ be a vertex operator algebra and 
$U=(U,Y,{\bf 1},\omega^1)$ be a vertex operator subalgebra of $W$ such that $L(1)\omega^1=0.$  The {\em commutant} $C_W(U)$ of $U$ is defined to be
$$C_W(U)=\{w\in W|u_nw=0, u\in U,n\geq 0\}.$$
Set $\omega^{2}=\omega-\omega^{1}.$ Then $(C_W(U), Y, \1,\omega^2)$ is also a vertex operator subalgebra of $V$ \cite{GKO,FZ}. Here is a characterization of $\overline{\CC_V}:$
 \begin{thm}\label{t4.2} Let $V,U$ be as before. Then  $\overline{\CC_V}$ and $\CC_U$ are braided equivalent if and only if there exists a holomorphic vertex operator algebra $W$ such that $V\otimes U$ is a conformal subalgebra of $W$ satisfying 
 $C_{W}(V)=U$ and $C_W(U)=V.$ 
 \end{thm}
 
 \begin{proof} It is well known that $\oplus_{X\in \Or(\CC_V)}X\otimes X'$ is a condensable algebra in the modular tensor category $\CC_V\otimes \overline{\CC_V}.$ If $\overline{\CC_V}$ and $\CC_U$ are braided equivalent, let ${\cal F}: \overline{\CC_V}\to \CC_U$ be a braided tensor functor giving the equivalence. Then $\Or(\CC_U)=\{{\cal F}(X')|X\in \Or(\CC_V)\}$ and
 $W=\oplus_{X\in \Or(\CC_V)}X\otimes {\cal F}(X')$ is a condensable algebra in $\CC_V\otimes \CC_U$, which is equivalent to $\CC_{V\otimes U}$ as braided tensor categories.
 It follows from \cite{HKL} that $W$ is a vertex operator algebra which is an extension of $V\otimes U.$ Since $\FPdim((\CC_{V\otimes U)})_W^{0})=1,$ one concludes immediately that $W$ is a holomorphic vertex operator algebra. It is clear that
 $C_{W}(V)=U$ and $C_W(U)=V$ from the construction of $W.$ 
 
 Now assume that there exists a holomorphic vertex operator algebra $W$ such that $V \ot U$ is a conformal subalgebra of $W$ satisfying $C_{W}(V)=U$ and $C_W(U)=V.$ Using a result in \cite{KM} we know that every irreducible $V$-module and $U$-module appear in $W$ as $W$ is holomorphic. 
 By Theorem 3.3 of \cite{Lin} we conclude that $\overline{\CC_{V}}$ and $\CC_{U}$ are braided equivalent.
 \end{proof}

 The following result asserts that Conjecture \ref{conjecture4.1} holds for lattice vertex operator algebra, which can also be obtained from Proposition \ref{pointed}.
 \begin{prop}\label{p4.3} Let $L$ be a positive definite even lattice, then $\overline{\CC_{V_L}}$ is braided equivalent to $\CC_{V_K}$ for some positive definite even lattice $K.$
 \end{prop}
 \begin{proof} First, the lattice vertex operator algebra $V_L$ \cite{B,FLM} is rational \cite{D1,DLM2}  and $C_2$-cofinite \cite{Z,DLM4}. So $\CC_{V_L}$ is a modular tensor category. If $L$ is also unimodular, $V_L$ is holomorphic and $\CC_{V_L}$
 is braided equivalent to $\Vec$ and $\overline{\CC_{V_L}}=\CC_{V_L}.$ Now we assume that $L$ is not unimodular. 
 By Theorem 5.5 of \cite{GH}, $L$ can be embedded in a positive definite even unimodular lattice $E$ and $L$ is a direct summand of $E$ as abelian groups and $O(L)$ embeds in $O(E)$ where $O(L)$ is the isometry group of $L.$ Then $V_E$ is a holomorphic vertex operator algebra. Let $K=L^{\perp}$ be the orthogonal complement of $L$ in $E.$ Then $V_{L\oplus K}=V_L\otimes V_K$ is a conformal subalgebra of $V_E$ in the sense that $V_{L\oplus K}$ and $V_E$ have
 the same Virasoro element.
 
 Let $C_{V_E}(V_L)$ be the commutant of $V_L$ in $V_E.$ We claim that $C_{V_E}(V_L)=V_K$ and $C_{V_E}(V_K)=V_L.$ Clearly,
 $V_K$ is a subalgebra of $C_{V_E}(V_L).$ Recall from \cite{D1} that the irreducible $V_K$-modules are given by $\{V_{K+\alpha_i}|i\in K^{\circ}/K\}$ where $K^{\circ}$ is the dual lattice of $K$ and $\alpha_i$ are the representatives of
 cosets of $K$ in $K^{\circ}.$ So $C_{V_E}(V_L)$ is a simple current extension of $V_K$ and there  is an even sublattice $K_1$ of $E$ containing $K$ such that $K_1$ and $K$ have the same rank, and $C_{V_E}(V_L)=V_{K_1}.$ This implies that $K_1$ is orthogonal to $L,$  $K_1=K$ and $C_{V_E}(V_L)=V_{K}.$ Similarly, $C_{V_E}(V_K)=V_{L_1}$ for an even sublattice $L_1$ of $E$ such that
 $L_1$ contains $L$ and $L_1,L$ have the same rank. Thus $L_1/L$ is a finite abelian group. The fact that $L$ is a direct summand 
 of $E$ as abelian groups forces $L_1=L.$ The result now follows from Theorem \ref{t4.2}
 \end{proof}

\begin{ex}
{\rm
Let $L_{A_1}=\Z\alpha$ with $(\alpha,\alpha)=2.$ Then $V_{L_{A_1}}$ has two irreducible modules $V_{L_{A_1}},V_{L_{A_1+\alpha/2}}.$ Let $K=L_{E_7}$
be the root lattice of type $E_7.$ Recall from  \cite{Hum} that the root lattice $L_{E_8}$  of type $E_8$ is spanned by
$\alpha_i$ for $i=1,...,8$ where $\alpha_1=\frac{1}{2}(\epsilon_1+\epsilon_8-(\epsilon_3+\cdots+\epsilon_7)),$  $\alpha_2=\epsilon_1+\epsilon_2,$ 
$\alpha_i=\epsilon_{i-1}-\epsilon_{i-2}$ for $i=3,...,8$ and $\{\epsilon_i| i=1,...,8\}$  is the standard  orthonormal basis of $\BR^8.$ Then $L_{E_7}$ can be identified with the sublattice  $\oplus_{i=1}^7\Z\alpha_i$ of $L_{E_8}$ and 
 $L_{A_1}$ can be identified with sublattice  $\Z(\epsilon_7+\epsilon_8)$ of $L_{E_8}.$ It is easy to see that $\oplus_{i=1}^7\Z\alpha_i$ and  $\Z(\epsilon_7+\epsilon_8)$ are orthogonal. 
 
 We claim that $L_{E_7}+L_{A_1}$ has index 2 in $L_{E_8}.$ Clearly, $\alpha_8=\epsilon_6-\epsilon_7$ does not lie in $L_{E_7}+L_{A_1}.$ So it is good enough to show that $2\alpha_8$ lies in $L_{E_7}+L_{A_1}.$ Observe that $2\epsilon_i$ belongs to
 $L_{E_7}+L_{A_1}$ for all $i.$ Thus $L_{E_8}=(L_{E_7}+L_{A_1})\cup (L_{E_7}+L_{A_1}+\alpha_8),$ as claimed. One can verify that
 $C_{V_{L_{E_8}}}(V_{L_{_{E_7}}})=V_{L_{A_1}}$ and  $C_{V_{L_{E_8}}}(V_{L_{_{A_1}}})=V_{L_{E_7}}$ by noting that $\alpha_8$ is not orthogonal to $L_{E_7}$  and $L_{A_1}.$ It is immediate from Theorem \ref{t4.2} that $\overline{\CC_{V_{L_{A_1}}}}$ and
 $\CC_{V_{L_{E_7}}}$ are braided equivalent.
 }
\end{ex} 

\section{$V^G$-modules }

In the rest of this  paper we assume the following:
\begin{enumerate}
\item[(V1)] $V=\oplus_{n\geq 0}V_n$ is a simple vertex operator algebra  of CFT type,
\item[(V2)] $G$ acts faithfully on $V$ as automorphisms such that $V^G$ is regular,
\item[(V3)] The conformal weight of any irreducible $V^G$-module $N$  is nonnegative and is zero if and only if $N=V^G.$
\end{enumerate}

If $\sigma \in \Aut(G)$, then the action of $G$ on $V$ can be twisted by $\sigma$, i.e., $g\cdot v = \sigma(g)v$. This twisted $G$ action on $V$ defines another automorphism group of $V.$ In general, if $G$ is an abstract group, it is possible to embed $G$ into $\Aut(V)$ in different ways. 

Assumption (V2) implies that $V$ is $C_2$-cofinite \cite{ABD} and $V$ is $g$-rational for all $g\in G$ \cite{ADJR}. The assumption
(V3) implies that 
the conformal weight of any irreducible $g$-twisted $V$-module except $V$ is positive, and that both $V^G$ and $V$ are selfdual.

We remark that if $G$ is solvable, then $V^G$ is regular if and only $V$ is regular.  For arbitrary $G$, the regularity of $V$ together with the $C_2$-cofiniteness of $V^G$ implies the rationality of $V^G$ \cite{Mc}.

From our assumptions,  both $\CC_{V}$ and $\CC_{V^G}$ are modular tensor category. 
Moreover,
$V$ is a condensable algebra in $\CC_{V^G}$ \cite{HKL}.  To simplify the notation, we use ${\rm Rep}(V)$ to denote the $V$-module category
$(\CC_{V^G})_V$ in $\CC_{V^G}$ \cite{KO}.  Then $\Rep(V)$ consists of every $V^{G}$-module
$W$ together with a $V^{G}$-intertwining operator $Y_W(\cdot,z)$ of type $\binom{W}{V\ W}$
such that the following conditions are satisfied:

1. (Associativity) For any $u,v\in V,$ $w\in W$ and $w'\in W'$, the formal series
\[
\langle w',Y_{W}(u,z_{1})Y_{W}(v,z_{2})w\rangle
\]
and
\[
\langle w',Y_{W}(Y(u,z_{1}-z_{2})v,z_{2}))w\rangle
\]
converge on the domains $|z_{1}|>|z_{2}|>0$ and
 $|z_{2}|>|z_{1}-z_{2}|>0$, respectively, to multivalued analytic functions
which coincide on their common domain.

2. (Unit) $Y_{W}({\bf 1},z)=\Id_{W}.$

It is proved in \cite{DLXY} that if $V^G$ satisfies conditions (V1)-(V3) then
$$\Rep(V)=\bigoplus_{g\in G}\Rep(V)_g$$
where $\Rep(V)_g$ is the $g$-twisted $V$-module category. Moreover, $\Rep(V)$ is a  fusion category \cite{KO}, \cite{CKM}, \cite{EGNO} with tensor product $\boxtimes_{\Rep(V)}.$ Furthermore, $\Rep(V)_1$ which is denoted by $\Rep(V)^0$ is exactly the
modular tensor category $\CC_V$ by \cite{HKL}.

We now discuss the connection between  the Frobenius-Perron dimension and the quantum dimension defined in \cite{DJX}, \cite{DRX} for a $g$-twisted
$V$-module $M=\oplus_{n\in\frac{1}{T}\Z_+} M_{\l+n}$ where  $T$ is the order of $g.$ Define the character of $M$
by
$$\chi_M(\tau)=q^{-c/24}\sum_{n\in\frac{1}{T}\Z_+}\dim M_{\l+n}q^{\l+n}$$
where $q=e^{2\pi i\tau}$ and $\tau$ lies in the upper half  plane. Then $\chi_M(\tau)$ is a modular function on a congruence subgroup \cite{Z}, \cite{DLN}, \cite{DR}. The quantum dimension of $M$ over $V$ is defined as
$$\qdim_V(M)=\lim_{\tau\to 0}\frac{\chi_M(\tau)}{\chi_V(\tau)}$$
which is always a positive algebraic number greater than or equal to 1. Note that $M$ is an object in the fusion category $\Rep(V).$
It turns out that $\qdim_V(M)=\FPdim(M).$

Our goal is to understand various fusion categories associated to $V^G.$ We first present a result on the classification of irreducible $V^G$-modules or
determine $\Or(\CC_{V^G})$ from \cite{DRX}.
For this purpose, we need the action of $G$ on $\Rep(V)$  \cite{DLM4}. Let $g, h$ be two automorphisms of $V$ with $g$ of finite order. If $(M, Y_M)$ is a $g$-twisted $V$-module, there is a  $h^{-1}gh$-twisted  $V$-module $(M\circ h, Y_{M\circ h})$ where $M\circ h\cong M$ as vector spaces and
\begin{equation*}
Y_{M\circ h}(v,z)=Y_M(hv,z)
\end{equation*}
for $v\in V.$
This defines a right action of $G$ on the twisted $V$-modules and on isomorphism classes of twisted $V$-modules. Similarly, we can define a left action of $G$ on the twisted $V$-modules and on isomorphism classes of twisted $V$-modules such that 
$h\circ M=M$ as vector spaces and $Y_{h\circ M}(v,z)=Y_M(h^{-1}v,z)$ for $v\in V.$ Then $G$ acts on $\Rep(V)$
as monoidal functors and $\Rep(V)$ is a braided $G$-crossed category \cite{Mc}.

If $g,h$ commute, $h$ clearly acts on the $g$-twisted modules.
 Denote by $\mathscr{M}(g)$ the equivalence classes of irreducible $g$-twisted $V$-modules and set $\mathscr{M}(g,h)=\{M \in \mathscr{M}(g)| M\circ h\cong M\}.$ Note from Theorem \ref{grational} that
 if $V$ is $g$-rational, both $\mathscr{M}(g)$ and $\mathscr{M}(g,h)$ are finite sets.
 For any $M\in \mathscr{M}(g,h),$ there is a $g$-twisted $V$-module isomorphism
\begin{equation*}
\phi(h) : M\to M\circ h.
\end{equation*}
The linear map $\phi(h)$ is unique up to a nonzero scalar. If $h=1$ we simply take $\phi(1)=\Id_M.$

Let $M=(M,Y_M)$ be an irreducible $g$-twisted $V$-module.
We define a subgroup $G_M$ of $G$ consisting of $h\in G$ such that $M\circ h$ and $M$ are isomorphic.
As we mentioned in Section 2 there  is a projective
representation $h\mapsto \phi(h)$ of $G_M$ on $M$ such that
$$
\phi(h)Y_M(v,z)\phi(h)^{-1}=Y_M(hv,z)
$$
for $h\in G_M$ and $v\in V.$ 
Let $\a_M$  be the corresponding 2-cocycle in $C^2(G,\C^{\times}).$
Then $\phi(h)\phi(k)=\a_M(h,k)\phi(hk)$
for all $h,k\in G_M$. We may assume $\alpha_M$ has finite order. That is, there is a fixed positive integer $n$ such that  $\alpha_M(h,k)^n=1$ for all $h,k\in G_M.$  Let $\C^{\a_M}[G_M]=\oplus_{h\in G_M} \C\bar h$ be the twisted group algebra with product $\bar h\bar k=\alpha_M(h,k)\ol{hk}.$ It is well known that $\C^{\a_M}[G_M]$ is a semisimple associative algebra. It follows that $M$ is a $\C^{\a_M}[G_M]$-module.

Let $\Lambda_{M}$ (which was denoted by $\Lambda_{G_M,\alpha_M}$ in \cite{DRX})  be the set of all irreducible characters $\lambda$
of  $\C^{\a_M}[G_M]$. Denote
the corresponding simple module by $W_{\l}.$
Using the fact  that $M$ is a semisimple
$\C^{\a_M}[G_M]$-module, we let $M^{\lambda}$ be the sum of simple
$\C^{\a_M}[G_M]$-submodules of $M$ isomorphic
to $W_{\l}.$ Then
$$M=\bigoplus_{\lambda\in \Lambda_{M}}M^{\lambda}.$$
Moreover, $M^{\lambda}=W_{\lambda}\otimes M_{\l}$
where $M_{\l}=\Hom_{\C^{\a_M}[G_M]}(W_{\l},M)$ is the multiplicity
of $W_{\l}$ in $M.$ As in \cite{DLM1}, we can,
in fact, realize $M_{\l}$ as a subspace of $M$ in the following way.  Let $w\in W_{\l}$
be a fixed nonzero vector. Then we can identify
$\Hom_{\C^{\a_M}[G_M]}(W_{\l},M)$ with the subspace
$$\{f(w) |f\in \Hom_{\C^{\a_M}[G_M]}(W_{\l},M)\}$$
 of $M^{\l}.$ This gives a decomposition

\begin{equation}\label{decom}
M=\bigoplus_{\lambda\in \Lambda_{M}}W_{\l}\otimes M_{\lambda}
\end{equation}
 and each $M_{\l}$ is a module for vertex operator subalgebra $V^{G_M}$-module.
Recall that the group $G$ acts on the set ${\cal S}=\cup_{g\in G}\mathscr{M}(g)$ and $ M\circ h$ and $M$ are isomorphic
$V^G$-modules for any $h\in G$ and $M\in {\cal S}.$ It is clear that the cardinality of the $G$-orbit $|M\circ G|$ of $M$ is
$[G:G_M].$ Let $J$ be the orbit representatives of $\S.$
Then we have the following results  \cite{DRX}, \cite{DJX}.

 \begin{thm}\label{mthm1} 
 Assume that $V$ satisfies (V1)-(V3).
 \begin{enumerate}
    \item[\rm (1)] 
 The set
 $$\{M_{\l}\mid M\in J,\,\l\in \Lambda_{M}\}$$
 gives a complete list of inequivalent irreducible $V^G$-modules. That is,
 any irreducible $V^G$-module is isomorphic to an irreducible $V^G$-submodule $M_{\l}$ for some $M\in J$ and  $\lambda\in \Lambda_M.$

 \item[\rm (2)] We have a relation between quantum dimensions
$$\qdim_{V^G}(M_\l)=\dim W_{\l}\cdot [G:G_M]\cdot \qdim_V(M)$$
 where $M$ is an irreducible $g$-twisted $V$-module
and $\lambda\in \Lambda_{M}.$ In particular,  $\Lambda_{V}=\Irr(G)$ is the set
of irreducible characters of $G$ and $\qdim_{V^G}(V_\l)=\dim W_{\l}$ for $\l\in \Irr(G)$.
\end{enumerate}

\end{thm}

\section{ Fusion categories $\FF_{V^G}$  and  $\EE_{V^G}$}

Let $\FF_{V^G}$ be the subcategory of $\CC_{V^G}$ generated by $V^G$-submodules of $V$-modules and $\EE_{V^G}$
the subcategory of  $\CC_{V^G}$ generated by $V^G$-submodules of $V.$  We prove in this section that  $\EE_{V^G}$  is equivalent to $\Rep(G)$ as braided tensor categories, and $\FF_{V^G}$ is a braided fusion subcategory of
$\CC_{V^G}$ such that the Mu\"ger center $\FF_{V^G}'$ of $\FF_{V^G}$ is exactly $\EE_{V^G}.$
  
 \begin{thm}\label{theorem2} $\EE_{V^G}$ is a fusion subcategory of $\CC_{V^G}$  equivalent to the symmetric fusion category $\Rep(G)$ via a canonical braided tensor functor  $F^{V,G}: \Rep(G) \to \EE_{V^G}$. In particular,  $\EE_{V^G}$ is a symmetric fusion category.
\end{thm}
\begin{proof} First we prove that $\EE_{V^G}$ is a braided fusion subcategory of $\CC_{V^G}.$ Since $\EE_{V^G}$ is a semisimple category and
each simple object is isomorphic to $V_{\l}$ for some $\l\in \Irr(G),$ it suffices to show that $V_\l\boxtimes V_\mu$ lies in  $\EE_{V^G}$ for
any  $\lambda,\mu\in \Irr(G).$ From \cite{T}, \cite{DRX} we know that
$$V_\l\boxtimes V_\mu=\sum_{\nu\in \Irr(G)}N_{\l,\mu}^{\nu}V_{\nu}$$
where the fusion rules $N_{\l,\mu}^{\nu}$ is given by the tensor product decomposition of $G$-module
$$W_\l\boxtimes W_\mu=\sum_{\nu\in \Irr(G)}N_{\l,\mu}^{\nu}W_{\nu}.$$
Thus, $\EE_{V^G}$ is closed under the tensor product and is a braided fusion subcategory of $\CC_{V^G}.$

Second, we establish that $\EE_{V^G}$ is a symmetric braided fusion category. Equivalently we need to show that
$$c_{V_{\mu}, V_{\l}}\circ c_{V_{\l},V_{\mu}}=\id_{V_{\l}\boxtimes V_{\mu}}$$
for any $\l,\mu.$  Since $\theta_{V_{\nu}}=1$ for all $\nu\in\Irr(G)$ we see that 
$$\id_{V_{\l}\boxtimes V_\mu}=\theta_{V_{\l}\boxtimes V_\mu}=c_{V_{\mu}, V_{\l}}\circ c_{V_{\l},V_{\mu}}\circ (\theta_{V_{\l}}\boxtimes \theta_{V_{\mu}})=c_{V_{\mu}, V_{\l}}\circ c_{V_{\l},V_{\mu}}.$$

Finally, we show that  $\EE_{V^G}$ is braided equivalent to the symmetric braided fusion category $\Rep(G)$. The categorical dimension $\dim(X)$  for an object $X$  in a spherical fusion category defined as the trace of the identity morphism of $X$. Under our assumption, $\dim(X)=\FPdim(X)$ always positive for any object $X$ in $\CC_{V^G}$ or $\Rep(V)$. The positivity of $\dim(V_\l)$ together with  the fact $\theta_{V_{\l}}=1$ implies that  $\EE_{V^G}$ is a Tannakian category. By \cite{De},  $\EE_{V^G}$ is braided equivalent to
$\Rep(H)$ for a finite group $H.$ The problem is we do not know why $H$ is isomorphic to $G.$ 

We now prove that  $\Rep(G)$  is braided equivalent to $\EE_{V^G}$ directly. For this purpose we  define a functor $F=F^{V,G}$ from $\Rep(G)$ to $\EE_{V^G}$ to be the composition functor 
$$F(X)=\Hom_{G}(X^*,V)=\bigoplus_{n\geq 0}\Hom_{G}(X^*,V_n)$$ where $X^*$ is the dual of $X.$  It is easy to see that $F(X)$ is a $V^G$-module
such that for $a\in V^G$ and $f\in F(X),$ $(Y(a,x)f)(u)=Y(a,x)fu$ for $u\in X^*$ as $X$ is a $G$-module. 
Now we prove that $F(X)$ lies in category $\EE_{V^G}.$
It is good enough to assume $X=W_{\lambda}$ for some $\lambda\in \irr(G).$ Note that $W_{\lambda}^*$ is isomorphic to $W_{\lambda^*}$ where $\lambda^*$ is the dual character of $\lambda.$ One can easily see that $F(X)$ is isomorphic to $V_{\lambda^*}$ and lies in  $\EE_{V^G}.$ Moreover, $\Hom_G(X^*,V_n)$ is the eigenspace of $L(0)$ 
with eigenvalue $n.$ 

We also need to deal with morphisms. Let $X,Y$ be two $G$-modules and $\alpha:X\to Y$ be a $G$-module morphism. Let $\alpha':Y^*\to X^*$
be the adjoint map. For $f\in F(X)$ define  $F(\alpha)(f)=f\alpha': F(X)\to F(Y)$. We assert that  $F(\alpha)$ is a $V^G$-module homomorphism.
Let $f\in F(X)$. For any $v\in Y^*$ and $a\in V^G$ we see that $$F(\a)(Y(a,x)f)(v)=Y(a,x)f\alpha'(v)=Y(a,x)F(\alpha)(f)(v).$$
So $F$ is a functor from $\Rep(G)$ to $\EE_{V^G}$. 
 
 Next we show that $F$ is a braided tensor functor. Let $X,Y$ be $G$-modules. We use the natural identification between $(X\otimes Y)^*$ with $Y^*\otimes X^*.$  For $f\in F(X)$ and $g\in F(Y),$ one can show that
 $${\cal Y}(f,x)g(v\otimes u)=Y(fu,x)gv$$
  for $v\otimes u\in Y^*\otimes X^*$ defines an intertwining operator of type  $\left(_{F(X),F(Y)}^{F(X\otimes Y)}\right).$  In fact, for any $a\in V^G$ and formal variables $x_0,x_1,x_2$ we have to prove that 
 \begin{align*}
 & x_{0}^{-1}\text{\ensuremath{\delta}}\left(\frac{x_{1}-x_{2}}{x_{0}}\right)Y(a,x_{1}){\cal Y}(f,x_{2})g-x_{0}^{-1}\delta\left(\frac{x_{2}-x_{1}}{-x_{0}}\right){\cal Y}(f,x_{2})Y(a,x_{1})g\\
 & =x_{2}^{-1}\delta\left(\frac{x_{1}-x_{0}}{x_{2}}\right)Y\left({\cal Y}\left(a,x_{0}\right)f,x_{2}\right)g
\end{align*} 
or 
  \begin{align*}
 & x_{0}^{-1}\text{\ensuremath{\delta}}\left(\frac{x_{1}-x_{2}}{x_{0}}\right)Y(a,x_{1})Y(fu,x_{2})gv-x_{0}^{-1}\delta\left(\frac{x_{2}-x_{1}}{-x_{0}}\right)Y(fu,x_{2})Y(a,x_{1})gv\\
 & =x_{2}^{-1}\delta\left(\frac{x_{1}-x_{0}}{x_{2}}\right)Y\left(Y\left(a,x_{0}\right)fu,x_{2}\right)gv.
\end{align*} 
This is obvious as the last identity is just the Jacobi identity in $V.$ Thus ${\cal Y}(f,z)g$ is an intertwining map of  type  $\left(_{F(X),F(Y)}^{F(X\otimes Y)}\right)$ for any $z\in\C^\times.$
By the universal mapping property of the tensor product
 $\boxtimes_{P(z)}$ we have a unique  $V^G$-module homomorphism 
 $$J_{X,Y}: F(X)\boxtimes_{P(z)} F(Y) \to F(X\otimes Y)=\Hom_G(Y^*\otimes X^*, V)$$ 
  characterized by
 $$\overline{J_{X,Y}}(f\boxtimes_{P(z)} g)(v\otimes u)=Y(fu,z)gv$$
 for any $u\in X^*, v\in Y^*.$  
 Then we have
   \begin{eqnarray*}
 \overline{J_{X,Y}}(f\boxtimes_{P(z)} g)(v\otimes u)&=&Y(fu,z)gv\\
&=&e^{zL(-1)}Y(gv,-z)fu \\
&=&e^{L(-1)z}\overline{J_{Y,X}}(g\boxtimes_{P(-z)} f)(u\otimes v)\\
&=&\overline{J_{Y,X}}(e^{L(-1)z}g\boxtimes_{P(-z)}f)(u\otimes v)\\
&=&\overline{J_{Y,X}}\overline{c_{F(X),F(Y)}}(f\boxtimes_{P(z)} g)(u\otimes v).
\end{eqnarray*}
In particular, 
$$\overline{J_{X,Y}}(f\boxtimes g)(v\otimes u)=\overline{J_{Y,X}}\overline{c_{F(X),F(Y)}}(f\boxtimes g)(u\otimes v).$$
Let 
$\pi_{X,Y}: X\otimes Y\to Y\otimes X$ be the natural braiding of the vector space tensor product. Then $F( \pi_{X,Y}) \overline{J_{X,Y}}(f\boxtimes g)(u\otimes v)= \overline{J_{X,Y}}(f\boxtimes g)(v\otimes u).$ That is,
$J_{X,Y}$  is a natural isomorphism such that the following commuting diagram holds
for objects $X,Y$ in $\Rep(G):$ 
\begin{equation*}
\begin{tikzcd}
F(X)\boxtimes F(Y) \arrow[r, "c_{F(X),F(Y)}"]  \arrow[d,"J_{X,Y}"]&F(Y)\boxtimes F(X)   \arrow[d,"J_{Y,X}"]\\
F(X\otimes Y) \arrow[r, "F(\pi_{X,Y})"] &F(Y\otimes X).
\end{tikzcd}
\end{equation*} 

Let $\e: V^G \to F^{V,G}(\underline{1})$ be the natural map $a \mapsto f_a$ where $f_a(1)=a$ for $a \in V^G$, where $\u1$ denote the trivial $G$-module $\C$. Then $\e$ is clearly a $V^G$-module isomorphism.
Now, We prove that $(F, J, \e)$ is a monoidal functor. That is, we need  to verify the following commuting diagrams
\begin{equation}\label{e1}
\begin{tikzcd}
F(X)\boxtimes (F(Y) \boxtimes F(Z))\arrow[r,"\id\otimes J_{Y,Z}"]  \arrow[d, "A_{F(X),F(Y),F(Z)}"] &F(X)\boxtimes  F(Y\otimes Z)\arrow[r,"J_{X, Y\otimes Z}"] & F(X\otimes  (Y\otimes Z))  \arrow[d, "F(a_{X,Y,Z})"] \\
(F(X)\boxtimes F(Y))\boxtimes F(Z) \arrow[r, "J_{X\otimes Y}\otimes \id"] & F(X\otimes Y)\boxtimes F(Z)  \arrow[r, "J_{X\otimes Y,Z}"]  &F((X\otimes Y)\otimes Z)\ \end{tikzcd}
\end{equation} 
\begin{equation}\label{e2}
\begin{tikzcd}
V^G \boxtimes F(X) \arrow[r,"\e \boxtimes \id"] \arrow[d,"l_{F(X)}"] & F(\u1) \boxtimes F(X) \arrow[d,"J_{\u1, X}"] \\
 F(X)  & F(\u1 \ot X) \arrow[l,"F(l_X)"] 
\end{tikzcd} \quad\text{and}\quad
\begin{tikzcd}
F(X) \boxtimes V^G \arrow[r,"\id \boxtimes \e"] \arrow[d,"r_{F(X)}"] &  F(X) \boxtimes F(\u1)  \arrow[d,"J_{X, \u1}"] \\
 F(X)  & F(X\ot \u1) \arrow[l,"F(r_X)"] 
\end{tikzcd} \quad
\end{equation}
where $a_{X,Y,Z}$, $r_X$, $l_X$ are respectively the associativity, left and right unit isomorphisms of vector spaces, and $l_F(X)$, $r_F(X)$ respectively denote the left and the right unit isomorphisms of $V^G$-mod.

Since the proofs of two commuting diagrams in (\ref{e2})
are similar, we only give a proof of the first commuting diagram.  Note that $l_{F(X)}: V^G \boxtimes F(X)\to F(X)$
is characterized by $l_{F(X)}(\b1\boxtimes g)=g$ for $g\in F(X).$ That is,
$l_{F(X)}(\b1\boxtimes g)(u)=g(u)$ for $u\in X^*.$ 
On the other hand, 
\begin{eqnarray*}
\begin{aligned}
(F(l_X)\circ J_{\u1,X}\circ \e\boxtimes \id)(\1\boxtimes g)(u)&=(J_{\u1,X}\circ \e\boxtimes \id)(\1\boxtimes g)(u\otimes 1)\\
&=J_{\u1,X}(f_{\1}\bo g) (u\otimes 1)\\
&=Y(\1,1)g(u)\\
&=g(u).
\end{aligned}
\end{eqnarray*}
That is, $F(l_X)\circ J_{\u1,X}\circ \e\boxtimes \id=l_{F(X)}.$

We now prove (\ref{e1}). Let $z_1, z_2>0$ and $\gamma$ be a path in $\BR^\times$  from $z_1$ to $z_2,$ then we have the following commuting diagram
 \begin{equation*}
\begin{tikzcd}
F(X)\boxtimes_{P(z_1)} F(Y) \arrow[r, "J_{X,Y}"]  \arrow[d,"T_\gamma"]&F(X\otimes Y)   \arrow[d,"\id"]\\
F(X)\boxtimes_{P(z_2)}F(Y) \arrow[r, "J_{X,Y}"] &F(X\otimes Y)
\end{tikzcd}
\end{equation*}  
by noting that 
\begin{equation*}
\begin{split}
(\overline{J_{X,Y}}\circ \overline{T_{\gamma}})(f \boxtimes_{P(z_1)}g)(v\otimes u)&=\overline{J_{X,Y}}I_{\boxtimes_{P(z_2)}}(f,e^{l(z_1)})g(v\otimes u)\\
&=\overline{J_{X,Y}}I_{\boxtimes_{P(z_2)}}(f,z_1)g(v\otimes u)\\
&=Y(fu,z_1)gv
\end{split}
\end{equation*}
where we have used the fact that $I_{\boxtimes_{P(z_2)}}$ only involves with integral powers of $z.$

 Let $z_1>z_2>z_1-z_2>0$ and $\gamma_i$ be as before for $i=1,2,3.4.$ So it is good enough to show the following diagram is commutative:
\begin{equation*}
\begin{tikzcd}
F(X)\boxtimes (F(Y) \boxtimes F(Z))  \arrow[r,"\id\boxtimes J_{Y,Z}"]  \arrow[d, "(\id\boxtimes_{ P(z_1) }T_{\gamma_2} )\circ T_{\gamma_1}"] &F(X)\boxtimes  F(Y\otimes Z) \arrow[r,"J_{X, Y\otimes Z}"]  \arrow[d,"\\T_{\gamma_1}"]&F(X\otimes  (Y\otimes Z))  \arrow[d, "\id"] \\
F(X)\boxtimes_{P(z_1)} (F(Y) \boxtimes_{P(z_2)} F(Z)) \arrow[r,"\id\otimes J_{Y,Z}"]  \arrow[d, "A_{z_1,z_2}"] &F(X)\boxtimes_{P(z_1)}  F(Y\otimes Z) \arrow[r,"J_{X, Y\otimes Z}"] 
 &F(X\otimes  (Y\otimes Z))  \arrow[d, "F(a_{X,Y,Z})"] \\
 (F(X)\boxtimes_{P(z_1-z_2)} F(Y)) \boxtimes_{P(z_2)} F(Z)) \arrow[r,"J_{X,Y}\otimes \id"]  \arrow[d, "T_{\gamma_3} \circ (T_{\gamma_4} \boxtimes_{P(z_2)}\id )"] &F(X\otimes Y)\boxtimes_{P(z_2)}  F( Z) \arrow[r,"J_{X\otimes Y, Z}"] \arrow[d,"T_{\gamma_3}"]
 &F((X\otimes  Y)\otimes Z)  \arrow[d, "\id"] \\
 (F(X)\boxtimes F(Y))\boxtimes F(Z)  \arrow[r, "J_{X\otimes Y}\otimes \id"]  &F(X\otimes Y)\boxtimes F(Z)  \arrow[r, "J_{X\otimes Y,Z}"] &F((X\otimes Y)\otimes Z)
 \end{tikzcd}
\end{equation*} 
From the discussion above,  the sub-diagrams involving the first two rows and the last two rows are commutative.

Now, we discuss the commutativity of the sub-diagram involving the second and third rows. Let $f\in F(X), g\in F(Y), h\in F(Z)$ and $u\in X^*,v\in Y^*, w\in Z^*.$  Then 
\begin{equation*}
\begin{split}
& \ \ \ (\overline{F(a_{X,Y,Z})}\circ\overline{J_{X,Y\otimes Z}}\circ \overline{\id\otimes J_{Y,Z}} )(f\boxtimes_{P(z_1)}(g\boxtimes_{P(z_2)}h))(w\otimes(v\otimes u))\\
&=Y(fu,z_1)Y(gv,z_2)hw, \quad \text{and}
\end{split}
\end{equation*}
\begin{equation*}
\begin{split}
& \ \ \ (\overline{J_{X\otimes Y, Z}}\circ \overline{J_{X,Y}\otimes\id} \circ \overline{A_{z_1,z_2}})(f\boxtimes_{P(z_1)}(g\boxtimes_{P(z_2)}h)(w\otimes(v\otimes u))\\
&=(\overline{J_{X\otimes Y, Z}}\circ \overline{J_{X,Y}\otimes\id})((f\boxtimes_{P(z_1-z_2)} g)\boxtimes_{P(z_2)}h)(w\otimes(v\otimes u))\\
&=Y(Y(fu,z_1-z_2)gv,z_2)hw.
\end{split}
\end{equation*}
That is the commutativity of  the diagram 
\begin{equation*}
\begin{tikzcd}
F(X)\boxtimes_{P(z_1)} (F(Y) \boxtimes_{P(z_2)} F(Z)) \arrow[r,"\id\otimes J_{Y,Z}"]  \arrow[d, "A_{z_1,z_2}"] &F(X)\boxtimes_{P(z_1)}  F(Y\otimes Z) \arrow[r,"J_{X, Y\otimes Z}"] 
 &F(X\otimes  (Y\otimes Z))  \arrow[d, "F(a_{X,Y,Z})"] \\
 (F(X)\boxtimes_{P(z_1-z_2)} F(Y)) \boxtimes_{P(z_2)} F(Z)) \arrow[r,"J_{X,Y}\otimes \id"]  &F(X\otimes Y)\boxtimes_{P(z_2)}  F( Z) \arrow[r,"J_{X\otimes Y, Z}"] &F((X\otimes  Y)\otimes Z) \,.
\end{tikzcd}
\end{equation*} 

Finally  we prove that $F$ is an  equivalence. It is clear that $F$ is injective on morphisms. Since $F(W_\lambda) \cong V_{\lambda^*}$ for any irreducible character $\lambda$, $F$ is essentially surjective and  $\dim_\C \Hom_G(X,Y) = \dim_\C \Hom_{V^G}(F(X), F(Y))$. Therefore, $F$ is bijective on morphism spaces, and  $F$ is an equivalence.  
\end{proof}

From Theorem \ref{theorem2}, $F^{V,G} : \Rep(G) \to \CC_{V^G}$ is an embedding for any vertex operator algebras $V$ satisfying the assumptions  (V1)-(V3), its image is equivalent to $\EE_{V^G}$.   We may simply denote $\Rep(G)$ by $\EE$ in the sequel, and the pair $(\CC_{V^G}, F^{V,G})$ defines a braided $\EE$-category. 

\begin{thm}\label{t5.2}  With the assumptions (V1)-(V3), we have
\begin{enumerate}
    \item[\rm (1)] 
 $\FPdim (\FF_{V^G})=o(G) \cdot \FPdim (\CC_V)$ and $\FPdim (\CC_{V^G})=o(G) \cdot \FPdim (\FF_{V^G}).$

 \item[\rm (2)]  $\FF_{V^G}$ is a braided fusion category.

 \item[\rm (3)]   $\FF_{V^G}'=\EE_{V^G}.$ That is the Mu\"ger center of $\FF_{V^G}$ is the symmetric fusion category $\EE_{V^G}$, and $(\FF_{V^G}, F^{V,G})$ is a nondegenerate braided $\EE$-category.
 \item[\rm (4)]  $\CC_{V^G}$ is a minimal modular extension of $\FF_{V^G}$. If $V$ is holomorphic, $\CC_{V^G}$ is a minimal modular extension of $\EE_{V^G}$ and is braided equivalent to ${\cal Z}(\Vec_G^{\alpha})$ for some $\alpha\in H^3(G,S^1).$

 \item[\rm (5)] 
 $\overline{\CC_{V^G}}$ is a minimal modular extension of $\overline{\FF_{V^G}}.$ If $V$ is holomorphic, $\overline{\CC_{V^G}}$ is a minimal modular extension of $\EE_{V^G}$ and is braided equivalent to ${\cal Z}(\Vec_G^{\bar{\alpha}})$ where $\bar\alpha=\alpha^{-1}.$
 \end{enumerate}
\end{thm}

\begin{proof} (1) Let $J_0$ be the orbit representatives consisting of irreducible $V$-modules. Then by Theorem \ref{mthm1},
$$ \FPdim (\FF_{V^G})=\sum_{M\in J_0}\sum_{\lambda \in \Lambda_{M}}\qdim_{V^G} (M_{\lambda})^2.$$
and
 $$\qdim_{V^G} (M_{\lambda})=[G:G_{M}]\dim (W_{\lambda}) \cdot \qdim_{V} (M),$$
 $$o(G_M)=\sum_{\lambda \in \Lambda_{M}} \dim (W_{\lambda})^2.$$
Thus,
\begin{align*}
 \FPdim (\FF_{V^G})
=&\sum_{M\in J_0}\sum_{\lambda \in \Lambda_{M}}
[G:G_{M}]^2 \dim (W_{\lambda})^2\cdot\qdim_{V} (M)^2\\
=&\sum_{M\in J_0}[G:G_M]^2 o(G_M)\qdim_{V} (M)^2\\
=&o(G) \sum_{M\in J_0}[G:G_M] \qdim_{V} (M)^2\\
=&o(G)\sum_{M\in J_0}\sum_{N\in G\cdot M}\qdim_{V} (N)^2\\
=&o(G)\cdot \FPdim (\CC_V).
\end{align*}
The identity $\FPdim (\CC_{V^G})=o(G)\cdot \FPdim (\FF_{V^G})$ follows from $\FPdim (\CC_{V^G})=o(G)^2\cdot \FPdim (\CC_V)$ \cite{DRX}.

(2) Since  $\FF_{V^G}$ is a subcategory of the modular tensor category $\CC_{V^G},$ it suffices to show that for any $X,Y$ in $\FF_{V^G},$
$X\boxtimes_{V^G} Y$ is also in  $\FF_{V^G}.$

Recall the fusion category   $\Rep(V)=\oplus_{g\in G}\Rep(V)_g$ from Section 4. There is an induction functor
$$\Ind_{V^G}^V: \CC_{V^G}\to \Rep(V)$$
such that $\Ind_{V^G}^V(X)=V\boxtimes_{V^G} X$ for any object $X$ in $ \CC_{V^G}$. It follows from  \cite{KO} that $\Ind_{V^G}^V$ is a tensor functor, and it has a right adjoint  $\Res^{V}_{V^G}: \Rep(V)  \to \CC_{V^G}$, which is  the restriction functor. In particular,  the following holds:

(i)  $\Hom_V(\Ind_{V^G}^V(X), M)$ and $\Hom_{V^G}(X, \Res^{V}_{V^G}(M))$ are naturally isomorphic for any $V^G$-module $X$ and $M\in \Rep(V),$ and 

(ii) $\Ind_{V^G}^V(W_1\boxtimes_{V^G} W_2)$ and $\Ind_{V^G}^V(W_1)\boxtimes_{\Rep(V)} \Ind_{V^G}^V(W_2)$ are naturally isomorphic for any $V^G$-modules $W_1, W_2$.

If $M$ is an irreducible $g$-twisted $V$-module and $\lambda\in\Lambda_{M}$ we claim that 
\begin{equation}\label{eq:iso}
    \Ind_{V^G}^V(M_{\lambda}) \cong \bigoplus_{N\in M\circ G}W_{\lambda}\otimes N=\bigoplus_{i=1}^{[G:G_M]} W_{\lambda}\otimes M\circ g_i
\end{equation}
as $V$-modules, where $\{g_1, \dots, g_{[G:G_M]}\}$ is a set of representatives of the right cosets of $G_M$ in $G$. Noting from Theorem \ref{mthm1}, for any irreducible twisted $V$-module $N$, we have $\Hom_{V^G}(M_{\lambda},N)=0$ if $N \not\in M \circ G$, 
     and $\Hom_{V^G}(M_{\lambda},N)\cong W_{\lambda}$ if $N \in M\circ G$. Using (i)
     immediately concludes the isomorphism in \eqref{eq:iso}.

Let $M,N$ be two irreducible $V$-modules and $\lambda\in \Lambda_M, \mu\in \Lambda_N$ we claim that  $M_{\lambda}\boxtimes_{V^G}N_{\mu}$ lies in
$\FF_{V^G}.$ First, for any $\lambda \in \Lambda_M$ and irreducible $g$-twisted $V$-module $X$ with $g \ne 1$, 
$$
\Hom_{\Rep(V)}(\Ind_{V^G}^V M_\lambda, X) \cong \Hom_{V^G}(M_\lambda, \Res^V_{V^G}X) = 0
$$
by (i) and Theorem \ref{mthm1}. Therefore, $\Ind_{V^G}^V M_\lambda \in \CC_V$. By (ii), for any $\mu \in \Lambda_N$, we have
$$
\Ind_{V^G}^V(M_{\lambda}\boxtimes_{V^G}N_{\mu}) \cong 
\Ind_{V^G}^V (M_{\lambda})\boxtimes_{\Rep(V)}\Ind_{V^G}^V(N_{\mu}) \in\CC_{V}.
$$
It follows from (i) that $\Hom_{V^G}(M_{\lambda}\boxtimes_{V^G}N_{\mu}, X)=0$ for any $g$-twisted $V$-module $X$ and $1\ne g\in G.$ This implies that $M_{\lambda}\boxtimes_{V^G}N_{\mu}$ lies in $\FF_{V^G}.$ Thus $\FF_{V^G}$ is a fusion subcategory of $\CC_{V^G}$.

(3)  We first prove that for any $\lambda\in \irr(G),$ $V_{\lambda}$ lies in $\FF_{V^G}'.$ and hence $\FPdim(\FF_{V^G}')\ge o(G).$ Equivalently, we need to show that
$$c_{V_{\lambda},M_{\mu}}\circ c_{M_{\mu},V_{\lambda}}=\id_{M_{\mu}\boxtimes V_{\lambda}}$$
for any irreducible $V$-module $M,$ $\lambda\in \irr(G)$ and $\mu\in \Lambda_M.$ It follows from \eqref{eq:iso} that 
$$
\Ind_{V^G}^V(V_\lambda) \cong W_\lambda \ot V 
$$
as $V$-modules. Therefore,
\begin{align*}
    \Ind_{V^G}^V(M_\mu \boxtimes_{V^G} V_\lambda) &  \cong \Ind_{V^G}^V(M_\mu) \boxtimes_{\Rep(V)} \Ind_{V^G}^V(V_\lambda) \\ 
    & \cong W_\lambda \ot \Ind_{V^G}^V(M_\mu)  \cong W_\lambda\otimes W_{\mu} \ot \bigoplus_{i=1}^{[G:G_M]}  M\circ g_i
\end{align*}
By (i) and Theorem \ref{mthm1}, we find  $M_{\mu}\boxtimes V_{\lambda}$ is isomorphic 
to a sum of some irreducible $V^G$-submodules of $M$ as $M\circ g_i$ and $M$ are isomorphic $V^G$-modules.
This implies that $\theta_M=\theta_{M_{\mu}}=\theta_{M_\mu\boxtimes V_{\lambda}}$ as complex numbers. Using the fact that $\theta_{V_{\l}}=1$ and relation 
$$\theta_{M_\mu\boxtimes V_\lambda}=c_{V_{\l}, M_{\mu}}\circ c_{M_{\mu},V_{\l}}\circ (\theta_{M_{\mu}}\boxtimes \theta_{V_{\l}}),$$ 
we conclude $c_{V_{\lambda},M_{\mu}}\circ c_{M_{\mu},V_{\lambda}}=\id_{M_{\mu}\boxtimes V_{\lambda}}.$

 As  ${\cal C}_{V^G}$ is modular, it follows from (\ref{3.1}) that
$$\FPdim (\CC_{V^G})=\FPdim(\FF_{V^G})\cdot \FPdim(C_{\CC_{V^G}}(\FF_{V^G})).$$
From (1) we know that 
$$\FPdim (\CC_{V^G})=\FPdim(\FF_{V^G})\cdot o(G)$$
which forces $\FPdim(C_{\CC_{V^G}}(\FF_{V^G}))=o(G).$ The fact that  $\EE_{V^G}$ is a full subcategory of $C_{\CC_{V^G}}(\FF_{V^G})$ and they have the same dimension
immediately implies that 
$$\FF_{V^G}'=C_{\CC_{V^G}}(\FF_{V^G})=\EE_{V^G}.$$ 

(4) By (1)-(3) and Lemma \ref{l3.1}, $\CC_{V^G}$ is a minimal modular extension of $\FF_{V^G}.$ If $V$ is holomorphic, then $\FF_{V^G} = \EE_{V^G}$ and  the statement follows from \cite[Thm. 4.22]{LKW1}. 

(5) By Theorem \ref{theorem2}, $\overline{\EE_{V^G}}=\EE_{V^G}$ is a subcategory of $\overline{\FF_{V^G}}.$ It follows from (1)-(4) that  $\overline{\CC_{V^G}}$ is a minimal modular extension of  $\overline{\FF_{V^G}}.$ If $V$ is holomorphic, then $\overline{\FF_{V^G}} = \EE_{V^G}$ and the statement follows from \cite[Thm. 4.22]{LKW1}. 
\end{proof}

Note that $\C[G]^*$ is a regular algebra in $\Rep(G)$, which is the dual $G$-module of $\C[G]$, and is a commutative associative algebra over $\C$ with a basis $\{e_a \mid a \in G\}$ of complete orthogonal idempotents given by $e_a(b)=\delta_{a,b}$. It is easy to see that $a \cdot e_b = e_{ab}$, and the product $\mu: \C[G]^*\otimes \C[G]^*\to \C[G]^*$ defined by $e_ae_b=\delta_{a,b}e_a$ is a $G$-module homomorphism. The unit map of this algebra is given by $i_{\C[G]^*}: \underline{1} \to \C[G]^*, 1 \mapsto 1_{\C[G]^*}=\sum_{a \in G} e_a$. Now, we can prove the following result:

\begin{thm}\label{t5.3} Let $V$ and $G$ be as before. Then
\begin{enumerate}
    \item[\rm (1)] 
 $F^{V,G}(\C[G]^*)$ is an algebra isomorphic to $V$ in category
$\EE_{V^G}.$

 \item[\rm (2)] 
 For any subgroup $H$ of $G$, $F^{V,G}(\C[G/H]^*)$ is a subalgebra of $F^{V,G}(\C[G]^*)$ isomorphic to $V^H$ in category $\EE_{V^G}.$
 \end{enumerate}
\end{thm}
\begin{proof} 

(1) For short we set $F=F^{V,G}$ in this proof. We identify $\C[G]^{**}$ with $\C[G]$ in the usual way as $G$-modules, which means $b(e_a) = e_{a}(b) =\delta_{a,b}$. Since $\C[G]$ is a free $G$-module generated by $1$,  for any $G$-module $W$,  we have the natural isomorphism  of vector spaces
$\Hom_G(\C[G], W) \cong W$, which implies
$\Hom_G(\C[G], W) = \{f_w \mid w \in W\}$ where $f_w(a) =aw$ for $a \in G$. In particular, we have $F(\C[G]^*)=\{f_v \mid v\in V\}.$ 

Let $U_n=\Hom(\C[G], V_n) = \{f_v\mid v\in V_n\}$ for $n\geq 0.$ 
We now  show that the algebra $U=\oplus_{n\geq 0}U_n=F(\C[G]^*)$ with product map $\mu_U= F(\mu)\circ J_{\C[G]^*, \C[G]^*}: U\boxtimes U\to U$ and unit map $i_U:=F(i_{\C[G]^*})\circ \e: V^G \to U$ is isomorphic to $V$ in $\EE_{V^G}.$ Note that the adjoint map $\mu': \C[G]\to \C[G]\otimes \C[G]$ of $\mu$ is determined  by $\mu'(a)=a \otimes a$ for any $a \in G$. Thus $F(\mu): F(\C[G]^*\otimes \C[G]^*)\to F(\C[G]^*)$  is given by $F(\mu)(f)=f\mu'$
for $f\in F(\C[G]^*\otimes \C[G]^*)$.  It follows from the braided tensor equivalence $F$ that $i_U:= F(i_{\C[G]^*}) \circ \e$ is the unit map of the algebra $U$ in $\EE_{V^G}$. For any $u, v\in V,$ and $a\in G$,
\begin{align*}
\mu_U (f_u\boxtimes f_v)(a) & =(F(\mu)\circ J_{\C[G]^*, \C[G]^*})(f_u\boxtimes f_v)(a)\\
& =Y(au,1)av=aY(u,1)v=f_{Y(u,1)v}(a) 
\end{align*}
where $f_{Y(u,1)v}$ is understood to be $\sum_{n\in\Z}f_{u_nv}.$
Therefore, $\mu_U(f_u\boxtimes f_v)=f_{Y(u,1)v}$. 
Recall from \cite{HKL} that $V$ is also algebra in $\EE_{V^G}$ with the algebra product map
$$\mu_V(u\boxtimes v)=Y(u,1)v.$$
One can define the $\C$-linear isomorphism $\phi: v\mapsto f_v$ for $v\in V$ from $V$ to $U$. Then $\phi$ is a $V^G$-module map by the definition of $U$ which satisfies $\mu_U \circ (\phi \boxtimes \phi) = \phi \circ  \mu_V$ and $i_U =\phi\circ i_V$, where unit map $i_V: V^G \to V$ of $V$ is the inclusion map.
In particular, $U$ is a vertex operator algebra isomorphic to $V.$ In fact, this can be seen directly that the vertex operator algebra structure on $U$ is given by  $Y(f_u,x)f_v=f_{Y(u,x)v}=\sum_{n\in\Z}f_{u_nv}x^{-n-1}$
for $u,v\in V.$ Since $F$ is a braided tensor equivalence, $U=F(\C[G]^*)$ is a regular algebra of $\EE_{V^G}$ isomorphic to $V$.

(2) For any subgroup $H$ of $G$,  $\iota: \C[G/H]^* \to \C[G]^*, e_{aH} \mapsto \sum_{h \in H} e_{ah}$ is an algebra embedding in $\EE$, where $e_{aH}(bH) = \delta_{aH, bH}$. Therefore, $F(\C[G/H]^*) \xrightarrow{F(\iota)}F(\C[G]^*)$ is an algebra embedding in $\EE_{V^G}$.
We also identify $\C[G/H]^{**}$ with $\C[G/H]$ as $G$-modules.
From (1) we see  that 
$$F(\C[G/H]^*)=\Hom_G(\C[G/H],V)=\{f_v|v\in V, f_v(ah)=f_v(a)\   \forall a\in G,  h\in H\}.$$
So $f_v\in F(\C[G/H]^*)$ if and only if $ahv=av$ for any $a\in G, h\in H.$ This forces $v\in V^H$, and hence $F(\C[G/H]^*)=\{f_v|v\in V^H\}.$ Recall from (1) that $\phi(v)=f_v$ for $v\in V$ is an algebra isomorphism from $V$ to $F(\C[G]^*).$ It is clear now that
the restriction of $\phi$ to $V^H$ gives an algebra 
isomorphism from $V^H$ to $F(\C[G/H]^*),$ as desired.
\end{proof}

\begin{rem} Theorem \ref{t5.3} gives a categorical interpretation of the Galois correspondence given in \cite{DM2}, \cite{HMT} that there is a bijection between the subgroups $H$ of $G$ and vertex operator subalgebras of $V$ containing $V^G$ by sending $H$ to $V^H.$
Combining with a result in \cite{HKL} we know that  the condensable algebras in $\EE_{V^G}$ are exactly $V^H$ for subgroups $H$ of $G.$ On the other hand, the condensable algebras in $\Rep(G)$ are given by $\C[G/H]^*$ for subgroups $H$ of $G$ \cite{KO}. It is easy  from Theorem \ref{t5.3} to see that $F(\C[G/H]^*)$ is isomorphic to $V^H.$ 
\end{rem} 

\section{The group $\MM_v(\EE)$ and $\MM_v(\EE)$-sets}

It has already been known from \cite{LKW1,LKW2}  that $\MM(\EE)$ is an abelian group.  Our goal is to understand this group structure from the point of view of vertex operator algebra. 

We need more notations and results on the braided fusion category.  Let $\CC$ and $\DD$ be braided fusion categories. We denote by $\CC\otimes \DD$ the Deligne tensor product of $\CC$ and $\DD$. Then $L_\CC=\oplus_{X\in {\cal O}(\CC)}X\otimes X^*$ is a contestable algebra in $\CC\otimes \overline{\CC}$  \cite{DMNO}. We also need a fact from \cite{KO} that the right adjoint of the tensor functor $\EE\otimes \EE \xrightarrow{\ot} \EE$ defines a braided tensor equivalence $R: \EE \to \EE \ot_\EE \EE = (\EE \ot \EE)_{L_\EE}$ and $R(\1)\cong L_\EE$ as algebras in $\EE \ot \EE$. Now let $(\CC,\iota_{\CC})$ and $(\DD,\iota_{\DD})$ be braided $\EE$-categories with embeddings $\iota_\CC: \EE\to \CC$ and $\iota_\DD: \EE\to \DD.$ 
Then 
 $$\iota_\CC\otimes_\EE\iota_\DD:  \EE \xrightarrow{R} (\EE\otimes \EE)_{L_\EE}\xrightarrow{\iota_\CC \ot \iota_\DD} (\CC\otimes \DD)_{A}^0$$
defines an embedding of $\EE$ into $(\CC\otimes \DD)_{A}^0$, where $A =(\iota_\CC \ot \iota_\DD)R(\1) \cong (\iota_\CC \ot \iota_\DD)(L_\EE)$. Following \cite{DNO}, \cite{LKW1} one can define the tensor product of braided $\EE$-categories as
\begin{equation} \label{eq:product}
    \CC\otimes_{\EE}^{(\iota_\CC,\iota_\DD)}\DD:=((\CC\otimes \DD)_{A}^0, \iota_\CC\otimes_\EE\iota_\DD)\,.
\end{equation}

Let $G$ be a finite group and  $V$ a vertex operator algebras satisfying conditions (V1)-(V3) such that $G$ acts faithfully on $V$ as automorphisms of $V$.  
We say that two such vertex operator algebras $V_1, V_2$ are \emph{equivalent} if there exists an isomorphism $\s : V_1 \to V_2$ of vertex operator algebras which commutes with their $G$-actions.  In this case, it is easy to see that  $V_1^{G}$ and $V_2^{G}$ are isomorphic and their module categories are braided equivalent.
So the number of inequivalent  faithful $G$-action on $V$ as automorphisms  is bounded by the cardinality of the conjugacy classes of $G$ in $\Aut(V).$ For example, there are two inequivalent faithful $\Z_2$ actions on the Moonshine vertex operator algebra $V^{\natural}$ \cite{FLM}. 

We denote by ${\bf R}_G$ for the the collection of vertex operators algebras satisfying conditions (V1)-(V3) with  $G$ acting faithfully as automorphisms. The subcollection of ${\bf R}_G$ consisting of holomorphic vertex operators algebras is denoted by ${\bf H}_G$.

The collection ${\bf H}_G$ of holomorphic vertex operators algebras could be generalized to non-holomorphic ones as follows: Fix a nondegenerate pseudounitary braided $\EE$-category $\FF$. Let
${\bf R}_G^\FF$ be the collection of vertex operator algebras $V\in {\bf R}_G$ such that $\FF_{V^G}$ is braided equivalent to $\FF$. The underlying braided equivalence $j^{V, G}: \FF_{V^G} \to \FF$ induces an $\EE$-braided equivalence $j^{V, G}: (\FF_{V^G}, F^{V,G}) \to (\FF, j^{V, G} \circ F^{V,G}).$ 
In particular, if $\FF=\EE$,  ${\bf R}^{\EE}_G = {\bf H}_G$.   We will use the notation $[\CC_{V^G}]$ to denote the equivalence class of the braided $\EE$-category $(\CC_{V^G}, F^{V,G})$ for any $V \in {\bf R}_G^\FF$.
By Theorem \ref{t5.2}, $\CC_{V^G}$
is a minimal modular extension of $\EE$ and for $V\in {\bf H}_G$ and $\CC_{U^G}$ is a minimal  modular extension of $\FF$ for $U\in {\bf R}_G^\FF.$ 

Set $\MM_{v}(\EE)=\{[\CC_{V^G}] \mid V\in {\bf H}_G\}$, and 
$\MM_{v}(\FF)=\{[\CC_{U^G}] \mid U\in {\bf R}^\FF_G\}$. 
By Theorem \ref{t5.2}, we have the inclusions $\MM_{v}(\EE)\subseteq \MM(\EE)$ and $\MM_{v}(\FF)\subseteq \MM(\FF).$

\begin{thm}\label{t7.1} Fix a finite group $G$ and let $\EE=\Rep(G)$. Then:
\begin{enumerate}
    \item[\rm (1)] The product 
    $$[\CC_{V^G}]\cdot [\CC_{U^G}]=[\CC_{(V\otimes U)^G}]$$
for $U,V\in {\bf H}_G$ on $\MM_{v}(\EE)$ coincides with 
    the product of $\MM(\EE)$ given by \eqref{eq:product}, 
 and hence $\MM_{v}(\EE)$ is a subgroup of 
$\MM(\EE)$. 
Moreover, $\MM_{v}(\EE)$ is isomorphic to a subgroup of $H^3(G, S^1)$, and $\CC_{V^G}$ is braided equivalent to ${\cal Z}(\Vec_G^{\alpha})$
for some $\alpha\in H^3(G,S^1)$ if $V\in {\bf H}_G.$
 \item[\rm (2)] For any (pseudounitary) braided  $\EE$-category $\FF$,  if ${\bf R}_G^\FF$ is not an empty collection, then the product  $[\CC_{V^G}]\cdot [\CC_{W^G}]:=[\CC_{(V\otimes W)^G}]$
for $V\in{\bf H}_G$ and $W\in  {\bf R}_G^\FF$ defines a free action of $\MM_{v}(\EE)$ on $\MM_{v}(\FF)$ with at most 
$[\M(\EE):\M_v(\EE)]$ many orbits.
 \item[\rm (3)] The group $\MM_{v}(\EE)$ acts freely 
 on $\{[\CC_{W^G}]|W\in {\bf R}_G\}$ such that 
 $[\CC_{V^G}]\cdot [\CC_{W^G}]:=[\CC_{(V\otimes W)^G}].$

\end{enumerate}
\end{thm}

\begin{proof} (1) 
 Let $W$ be any holomorphic vertex operator algebra. Then there exists a positive integer $n$ such that $G$ can be realized as a subgroup of the symmetric group $S_n$. Thus $G$ is an automorphism group of holomorphic vertex operator algebra $V=W^{\otimes n}.$ That means $\CC_{V^G}$ lies in $\MM_{v}(\EE)$ and so  $\MM_{v}(\EE)$ is not empty. 
 
All the statements in the theorem is a consequence of the equality:  
$$
[\CC_{V^G}]\cdot [\CC_{U^G}]=[\CC_{(V\ot U)^G}]
$$
for any $V \in \H_G$ and $U \in {\bf R}_G$. It amounts to prove that 
\begin{equation}\label{7.1}
(\CC_{(V\otimes U)^G}, F^{{V\otimes U, G}})\cong \CC_{V^G}\otimes_\EE^{(F^{V,G},F^{U,G})} \CC_{U^G},
\end{equation}
as braided $\EE$-categories, which means we need to find a braided tensor equivalence $\tilde\phi : (\CC_{V^G} \ot \CC_{U^G})_A^0 \to \CC_{(V \ot U)^G}$ such that $\tilde\phi \circ (F^{V,G}\ot_\EE F^{U,G})\cong F^{V \ot U, G}$ as  tensor functors, where $A =(F^{V,G} \ot F^{U,G})(L_\EE)$.

To find such a braided tensor equivalence $\tilde\phi$,   we first show that $(V \ot U)^G$  and $A$   are isomorphic algebras in $\CC_{V^G} \ot \CC_{U^G}$ under an isomorphism $\phi$. This algebra isomorphism induces a (strict)  tensor equivalence $\tilde\phi : (\CC_{V^G} \ot \CC_{U^G})_A \to (\CC_{V^G} \ot \CC_{U^G})_{(V \ot U)^G}$ and hence a braided tensor equivalence $\tilde\phi: (\CC_{V^G} \ot \CC_{U^G})^0_A \to (\CC_{V^G} \ot \CC_{U^G})^0_{(V \ot U)^G}=\CC_{(V \ot U)^G}$. 

By Theorem \ref{t5.3}, there exists an algebra isomorphism  
$$
V \ot U \xrightarrow[\phi]{\sim} F^{{V\otimes U, G \times G}}(\C[G \times G]^*)
$$
in $\EE_{(V \ot U)^{G \times G}}$. Consider $G$ as the diagonal subgroup of $G \times G$. Then  $\C[(G \times G)/G]^*$ is a subalgebra of $\C[G \times G]^*$ in $\Rep(G \times G)$. It follows from Theorem \ref{t5.3} (2) that the restriction of $\phi$ defines an isomorphism 
$$
(V \ot U)^G \xrightarrow[\phi]{\sim} F^{{V\otimes U, G \times G}}(\C[(G \times G)/G]^*)
$$
in $\EE_{{(V \ot U)^{G \times G}}}$. Note that if one identifies $\Rep(G \times G)$ with 
$\EE \ot \EE$, then $\C[(G \times G)/G]^* = L_\EE$ and $F^{{V\otimes U, G \times G}} = F^{V,G} \ot F^{U,G}$. Therefore, $\phi$ defines an isomorphism of algebras in $\CC_{V^G} \ot \CC_{U^G}$ from  $(V\ot U)^G$ to $A.$ Since $A$ is an algebra in $\EE_{V^G} \ot \EE_{U^G}$, so is $(V\ot U)^G$.

Now, the algebra isomorphism $\phi$ in $\CC_{V^G} \ot \CC_{U^G}$ induces a  braided tensor equivalence  $\tilde \phi:  (\CC_{V^G}\otimes \CC_{U^G})_{A}^0 \to \CC_{(V\otimes U)^G}$, and $\phi: (V \ot U)^G \to \tilde \phi(A)$ is an isomorphism of $(U \ot V)^G$-modules, and  $(\tilde \phi \circ (F^{V,G} \ot_\EE F^{U,G}), J', \phi)$ is a braided tensor functor, where $J'= J^{V \ot U,G \times G}.$ To complete the proof of \eqref{7.1}, we need to show that $(\EE_{V^G}\otimes\EE_{U^G})_{(V\otimes U)^G} = \EE_{(V\otimes U)^G}$ and  the equivalence of the two braided tensor functors:
$$(\tilde \phi \circ (F^{V,G} \ot_\EE F^{U,G}), J', \phi) \cong (F^{{V\otimes U, G }}, J^{V\ot U, G}, \e).$$
where $J^{V \ot U, G}$ and $\e$ are the monoidal structure of the functor $F^{V \ot U, G}$ defined  in Theorem \ref{theorem2}.

Note that 
the simple $(V \ot U)^G$-submodules $(V \ot U)_\lambda \in (\EE_{V^G} \ot \EE_{U^G})_{(V\otimes U)^G}$ for $\lambda \in \Irr(G)$. Since $\FPdim ((\EE_{V^G} \ot \EE_{U^G})_{(V\otimes U)^G}) = |G|$, we find $(\EE_{V^G}\otimes\EE_{U^G})_{(V\otimes U)^G} = \EE_{(V\otimes U)^G}$. 

 Using the identification 
$\EE\otimes \EE=\Rep(G\times G)$, we can write each  object 
$X\in \EE\otimes_\EE \EE$ as a $G \times G$-module with a decomposition $X=\oplus_{x\in G\times G/G}X_x$ 
such that $(g,h)X_x=X_{(g,h)x}$ for $(g,h) \in G \times G$  \cite{KO}. In particular, $X_1$ is a $G$-module, where $1$ denotes the diagonal subgroup of $G \times G$. Moreover, the induction functor $\Ind_G^{G \times G} : \EE \to \EE\otimes_\EE \EE$ is the corresponding right adjoint of the braided equivalence  $\EE \ot_\EE \EE \xrightarrow{\otimes} \EE$. In this convention, $\tilde\phi \circ (F^{V,G} \ot F^{U,G}) \circ \Ind_G^{G \times G} \stackrel{\psi}\cong  F^{V \ot U, G}$, where $\psi$ is given by  the natural isomorphism,
\begin{align*}
   (\tilde \phi \circ (F^{V,G} \ot F^{U,G}) \circ \Ind_G^{G \times G})(X) & = \Hom_{G \times G}((\Ind_G^{G \times G}(X))^*, V \ot U)  \\
    & \cong  \Hom_{G \times G}(\Ind_G^{G \times G}(X^*), V \ot U)  \\
    & \cong  \Hom_{G}(X^*, V \ot U) \\
    & =F^{V \ot U, G} (X)
\end{align*}
for $X \in \EE$. With the identification of $X^*$ and $(\Ind_G^{G \times G} (X))^*_1$, the inverse $\psi^{-1} : F^{V \ot U, G} (X) \to \Hom_{G \times G}((\Ind_G^{G \times G}(X))^*, V \ot U)$ is given by
$$\psi^{-1}(f)(yu)=yf(u)$$
for any $u \in X^*$, $y \in G \times G$ and $f \in F^{V \ot U, G} (X)$.  In particular, $\psi^{-1}(f)(yu) = f(yu)$ if $y \in G$.
Now  we need to show $\psi$ is a isomorphism of the tensor functors which requires to prove the commutativity of following diagrams for $X,Y\in \EE$ :
\begin{equation*}
\begin{tikzcd}
F^{V \ot U, G}(X)\boxtimes F^{V \ot U, G}(Y) \arrow[r, "J^{V \ot U, G}_{X,Y}"]  \arrow[d,"\psi^{-1}\boxtimes \psi^{-1}"]& F^{V \ot U, G}(X\otimes Y)\arrow[d,"\psi^{-1}"] \\
F'(X)\boxtimes F'(Y) \arrow[r, "J'_{X,Y}"] &F'(X\otimes Y) 
\end{tikzcd} \text{ and }
\begin{tikzcd}
(V \ot U)^G \arrow[r, "\e"] \arrow[rd, "\phi"'] & F^{V \ot U, G}(\uline{1}) \arrow[d, "\psi^{-1}"] \\
 & F'(\uline{1})
\end{tikzcd}
\end{equation*}
where $F'=\tilde\phi \circ (F^{V,G} \ot F^{U,G}) \circ \Ind_G^{G \times G}.$ Let $f\in F^{V \ot U, G}(X), g\in F^{V \ot U, G}(Y)$
and $u\in X^*, v\in Y^*.$  We know from the proof of Theorem \ref{theorem2} that $\psi^{-1}\circ J_{X,Y}^{V\ot U, G}$
is characterized by
 $$( \overline{\psi^{-1}\circ J_{X,Y}^{V\ot U, G}})(f\boxtimes g)(v\otimes u)=Y(fu,1)gv.$$
Since $J'=J^{V\otimes U, G\times G}=J^{V,G}\ot J^{U,G}$ we immediately see that $J'_{X,Y}$ is characterized by
$$\overline{J'_{X,Y}}(\psi^{-1}(f)\boxtimes \psi^{-1}(g))(v\otimes u)=Y(\psi^{-1}(f)u,1)\psi^{-1}(g)v=Y(fu,1)gv,$$ 
which proves the  commutativity of the first diagram. Note that the $(V \ot U)^G$-module isomorphism $\phi: (V \ot U)^G \to F'(\uline{1})$ is unique up to a scalar. Since $\psi^{-1}(\e (x \ot y))(1) = x \ot y =\phi(x \ot y)(1) $ for any $x \in V^G$ and $y \in U^G$, the second commutativity follows.

Therefore,  $\MM_{v}(\EE)$ is closed  under the product of  $\MM(\EE)$, and hence  $\MM_{v}(\EE)$ is a subgroup of  $\MM(\EE)$. Following the preceding discussion, there exists a unique $\a \in H^3(G, S^1)$ such that $(\CC_{V^G}, F^{V,G})$ is equivalent to $(\ZZ(\Vec^\a_G), \iota_\a)$. In particular, 
$\CC_{V^G}$ is equivalent to  $\ZZ(\Vec^\a_G)$ as modular tensor categories.
 
(2)  From the proof of (1), for $V \in {\bf H}_G$ and $W \in {\bf R}_G^\FF$ for some (pseudounitary) nondegenerate braided $\EE$-category $\FF$,  $[\CC_{W^G}] \in \MM_v(\FF)$  and  the pair $(\CC_{(V\ot W)^G}, F^{V\ot W, G}) \cong \CC_{V^G}\ot_\EE^{(F^{V,G},F^{W,G})}\CC_{W^G}.$ According \cite{LKW1}, $\CC_{V^G}\ot_\EE^{(F^{V,G},F^{W,G})}\CC_{W^G}$ is a minimal modular extension of $\FF$, so is $(\CC_{(V\ot W)^G}, F^{V\ot W, G})$.   Therefore, $V \ot W \in {\bf R}_G^\FF$ and $\MM_v(\FF)$ admits a $\MM_v(\EE)$-action  defined by $[\CC_{V^G}][\CC_{W^G}] :=[\CC_{(V \ot W)^G}]$  which coincides with $\MM(\EE)$-action on $\MM(\FF)$.   By \cite{LKW1}, $\MM(\FF)$ is an  $\MM(\EE)$-torsor. So, the 
 action of $\MM_{v}(\FF)$ on $\MM_{v}(\FF)$  is free. Since 
 the cardinality $|\MM(\FF)|$ is equal to $o(\MM(\EE)),$ it is immediate 
 that the number of $\MM_{v}(\EE)$-orbits on $\MM_{v}(\FF)$
 is less than or equal to the index $[\M(\EE):\M_v(\EE)].$
 
 (3) follows directly from (2).\end{proof}
 \begin{rem}
 We now explain how to associate 3-cocycle $\alpha\in Z^3(G,S^1)$ to a holomorphic vertex operator algebra $V$ satisfying conditions (V1)-(V3)  such that $\CC_{V^G}$ and $\ZZ(\Vec_G^\alpha)$ are braided equivalent. By Theorem \ref{t5.3}, $V$ is a condensable algebra in $\CC_{V^G}.$ Since $V$ is holomorphic, for every $g\in G$ there is a unique irreducible $g$-twisted $V$-module $V(g)$ up to equivalence \cite{DLM4}. 
According to \cite{DLXY}, every simple object in $(\CC_{V^G})_V$ is isomorphic to $V(g)$ for some $g\in G.$ From the discussion in Section 4, $(\CC_{V^G})_V$ is $G$-graded fusion category such that
 $((\CC_{V^G})_V)_g$ is generated by $V(g).$ The associativity isomorphism 
 $$(V(g)\boxtimes_{\Rep(V)} V(h))\boxtimes_{\Rep(V)}V(k)\to  V(g)\boxtimes_{\Rep(V)}(V(h)\boxtimes_{\Rep(V)}V(k)) $$
 determines an $\alpha\in H^3(G,S^1).$ 
 \end{rem}
 
 \begin{rem} It is definitely desirable that $\M(\EE)=\M_v(\EE)$,  but we could not supply a proof to this claim. If this is true and $\M_v(\FF)\ne \emptyset$ then $\M_v(\FF)=\M(\FF)$ is an $\M_v(\EE)$-torsor.
 \end{rem}
 
 \begin{rem} It is worthy to mention that for any $V\in {\bf H}_G,$ the inverse of $[\CC_{V^G}]$ in $\MM_{v}(\EE)$ is $[\overline{\CC_{V^G}}].$ By  Theorem \ref{t7.1}, $\overline{\CC_{V^G}}$ is braided equivalent to
 $\CC_{(V^{\otimes{m-1}})^G}$ where $m$ is the order of $[\CC_{V^G}].$ So Conjecture \ref{conjecture4.1} holds for rational vertex operator algebra $V^G$ for $V\in {\bf H}_G.$ 
 \end{rem}

 It is proved in \cite{EG} that if $G$ is solvable, then for any $\alpha\in H^3(G,S^1)$ there is a regular vertex operator algebra $V$ such that
 $\CC_V$ is braided equivalent to $\ZZ(\Vec^\a_G).$ But it is not clear to us that this $V$ can be realized as $U^G$ for some holomorphic vertex operator algebra $U$ with a faithful $G$-action. In the next section 
 we give a proof when $G$ is  a dihedral group or abelian group with at most two generators.
 
 \section{Lattice vertex operator algebras and pointed minimal extensions}
 
 We explain in this section how to use lattice vertex operator algebras to realize the pointed modular categories. In particular, if  $\ZZ(\Vec^\a_G)$ is pointed for some $\a \in H^3(G,S^1)$,
  we prove that there exists a positive definite even unimodular lattice $L$ such that $G$ can be realized as an  automorphism group of the lattice vertex operator algebra $V_L,$  $V_L^G$ is also a lattice vertex operator algebra and  $(\ZZ(\Vec^\a_G),\iota_\a)\cong (\CC_{V_L^G}, F^{V_L,G})$ as minimal modular extension of $\Rep(G)$.   
 
 For this purpose we need  to recall the construction of the vertex operator algebra $V_L$ associated to a  positive definite even lattice $L$ with a bilinear form 
 $(\cdot , \cdot)$ and its irreducible modules following \cite{D1, DL1, FLM}.  As usual we denote by $L^\circ$ the dual lattice of $L.$ Then there exists a positive even integer $m$
 and an alternating $\Z$-bilinear function
 $$c: L^{\circ}\times L^{\circ }\to \<\zeta_m\>$$
 such that $c(\alpha,\beta)=(-1)^{(\alpha,\beta)}$ for $\alpha,\beta\in L$   where $\zeta_m=e^{2\pi i/m}$ 
 (see \cite[Remark 12.18]{DL1}). In fact, $c(\alpha,\beta)=\e(\a,\beta)\e(\beta,\alpha)^{-1}$ for 
 some $\Z$-bilinear function $\e: L^{\circ}\times L^{\circ }\to \<\zeta_m\>$  for $\alpha,\beta\in L^{\circ}.$ 
 Consider the corresponding central extension $\widehat{L^{\circ}}$ of $L^{\circ}$ by the cyclic group $\langle\zeta_m\rangle$:
\[
1\rightarrow \langle\zeta_m\rangle\rightarrow \widehat{L^{\circ}}\mathop{\rightarrow}\limits^{-}L^{\circ}\rightarrow 0
\]
with commutator map $c(\cdot\,,\cdot).$ Let $e:L^\circ \to\widehat{L^\circ},\,\lambda\mapsto e_{\lambda}$ be a section such that $e_0=1$ and  $e_{\alpha}e_{\beta}=\epsilon(\alpha,\beta)e_{\alpha+\beta}$ for any $\alpha,\,\beta\in L^{\circ}$. 
We can assume that $\epsilon(\alpha,\alpha)=(-1)^{\frac{(\alpha,\alpha)}{2}}$
for any $\alpha\in L.$ Then the twisted group algebra $\C^\e[L^{\circ}]=\sum_{\alpha\in L^{\circ}}\C e^{\alpha}$ with product $e^\alpha\cdot e^{\beta}=\epsilon(\alpha,\beta)e^{\alpha+\beta}$
for $\alpha,\beta\in L^{\circ}$ is a quotient of the group algebra $\C[\widehat{L^{\circ}}]$ by identifying 
$\zeta_m\in \widehat{L^{\circ}}$ with $\zeta_m\in\C.$ 
It is easy to see that  $\C^\e[L^{\circ}]=\oplus_{i\in L^\circ/L}\C^\e[L+\lambda_i]$ where 
$\C^\e[L+\lambda_i]=\oplus_{\alpha\in L}\C e^{\lambda_i+\alpha}$ and $L^{\circ}/L=\{L+\lambda_i\mid i\in L^{\circ}/L\}.$

Let ${\frak h}=\C\otimes_\Z L$ be an abelian Lie algebra and extend the form $(\cdot,\cdot)$ to  ${\frak h}$ by $\C$-linearity.  Let $\{h_1,...,h_d\}$ be an orthonormal basis of $\frak h$ where $d$ is the rank of $L.$ Then the affine Lie algebra
 $$\widehat{\frak h}={\frak h}\otimes \C[t,t^{-1}]\oplus \C k$$
 has a unique irreducible module 
 $$M(1)=[h_i(-n)|i=1,...,d, n>0]$$
 such that $h_i(n)$ acts as $n\frac{\partial}{\partial h_i(-n)}$ if $n>0,$ as multiplication operator $h_i(n)$
 if $n<0$ and $0$ if $n=0,$ and $k$ acts as $1$ where $h(n)=h\otimes t^n$ for $h\in {\frak h}$ and $n\in\Z.$
 Set  $$V_{L^{\circ}}=M(1)\otimes \C^\e[L^{\circ}]=\bigoplus_{i\in L^{\circ}/L}V_{L+\lambda_i}$$ where 
$V_{L+\lambda_i}=M(1)\otimes \C^\e[L+\lambda_i].$ Then
$V_L$ is a vertex operator algebra and $\{V_{L+\lambda_i}\mid i\in L^{\circ}/L\}$ is a complete list of inequivalent irreducible $V_L$-modules.  Moreover, the ribbon structure of $\CC_{V_L}$ is given by 
$\theta_{V_{L+\lambda_i}}=e^{\pi i(\lambda_i,\lambda_i)}.$
For any $\alpha\in \Q\otimes_\Z L$, one can define an automorphism 
$\sigma_\a$ of  finite order of $V_L$ by setting 
$$\sigma_\a(u\otimes e^\beta)=e^{2\pi i(\alpha,\beta)}u\otimes e^\beta$$
 for $u\in M(1)$ and $\beta\in L.$ This type of  automorphism 
 will be useful in the following  discussions.
 
 We first show  that $\M_v(\EE)\cong H^3(G,S^1)$ when
 $G$ is a cyclic group or a  dihedral group by using 
 concrete lattice vertex operator algebras associated
 to the  Niemeier lattice
 $L$ of type $A_1^{24}.$ 
 
\begin{prop}\label{p7.2} Let $V,G, \FF, \EE$ be as before. If $G\cong \Z_n$ is a cyclic group or $G\cong D_{2m}$ is a dihedral  group of order $2m$ with $m$ being odd, then $\M_v(\EE)=\M(\EE)\cong H^3(G,S^1)$ and 
  $\MM_v(\FF)$ is a $\M_v(\EE)$-torsor if $\MM_v(\FF)$ is not empty.
  \end{prop}
 \begin{proof} 
 (1) $G\cong \Z_n.$ In this case  $H^3(G,S^1) \cong \Z_n$.  Therefore, it suffices to show that $\MM_{v}(\EE)$ has an element of order $n.$ Consider the holomorphic lattice 
 vertex operator algebra $V_L$ associated to the Niemeier lattice
 $L$ of type $A_1^{24}$ \cite{FLM}.  Then 
 $$L=\sum_{C\in G_{24}}\Z \frac{1}{2}\alpha_C+Q$$
 where $G_{24}$ is the Golay code based on the set $\Omega=\{1,...,24\},$ $Q=\sum_{i=1}^{24}\Z\alpha_i$ is a positive definite lattice with $(\alpha_i,\alpha_j)=2\delta_{i,j}$ and $\alpha_C=\sum_{i\in C}\alpha_i.$  
 Let $G$ be the cyclic group generated by $\sigma=\sigma_{\alpha_1/n}.$ Then $V_L^G=V_E$ where $E$ is the sublattice of $L$ given by $E=\{\alpha\in L\mid (\alpha_1,\alpha)\in n\Z\}.$ 
 Moreover, there is a  unique 
 irreducible $\sigma$-twisted module $V_L(\sigma)=M(1)\otimes \C^{\e}[L]e^{-\alpha_1/n}$ \cite{DM1}.  
 Note that $V_{E-\alpha_1/n}=M(1)\otimes \C^{\e}[E]e^{-\alpha_1/n}$ is an irreducible $V_E$-module
 with $\theta_{V_{E-\alpha_1/n}}=e^{2\pi i/n^2}.$ In particular, the order of $\tilde t=\diag(\theta_{M} \mid {M\in {\cal O}(C_{V_L^G})})$, denoted by $\FSexp(\CC_{V^G_L})$,   is a multiple of $n^2$.
 
By Theorem \ref{t7.1}, $(C_{V_L^G}, F^{V_L, G}) \cong (\ZZ(\Vec_G^\w), \iota_\w)$ as braided $\EE$-categories for some $\w \in H^3(G, S^1)$. 
By \cite{NS}, $\FSexp(\ZZ(\Vec_G^\w))$ is equal to least common multiple of $\ord(\w|_H) \cdot o(H)$, where $\w|_H$ is the restriction of $\w$ on $H$, and $H$ runs through the maximal cyclic subgroups of $G$.  It follows from  the preceding paragraph that  $\ord(\w)=n$. Therefore, $[\CC_{V_L^G}] \in \MM_v(\EE)$ is of order $n,$ as desired.  

(2) $G\cong D_{2m}$ for some odd integer $m$. Then $H^3(G,S^1)\cong\Z_{2m}.$ Recall from \cite[Theorems 10.1.2, 10.1.5]{FLM} that the Golay code $G_{24}$ is built from the Hamming codes $\CC_1$ and $\CC_2.$ 
In fact,
$$G_{24}=\<(S,S, \emptyset), (S,\emptyset, S), (T,T,T)|S\in \CC_1,T\in\CC_2\>.$$
Let $\Omega=\{1,....,24\}$ and  $w=\{1,2,11,12\}.$ Then
$|w\cap C|$ is always even for any $C\in G_{24}.$
 Define an isometry $\nu$ of $L$ such that 
$$\nu (\sum_{i=1}^{24} k_i\alpha_i)= -\sum_{i\in w}k_i\alpha_i+
\sum_{i\notin w}k_i\alpha_i$$
where $k_i\in \frac{1}{2}\Z.$  Then $(\nu(\alpha),\alpha)\in 2\Z$ for all $\alpha\in L.$
So $\nu$ satisfies the assumption in \cite{Le} and can be lifted to an automorphism $\tau$ of $V_L$ of order $2.$  Let $V_L(\tau)$ be the unique irreducible $\tau$-twisted $V_L$-module. Then 
$V_L(\tau)$ has a gradation $V_L(\tau)=\oplus_{n\geq 0}V_L(\tau)_{\frac{1}{4}+\ha n}$ as the eigenspace of
$\nu$ on ${\frak h}$ with eigenvalue $-1$ has dimension $4$ \cite{DL2}. Recall that  $\sigma_m=\sigma_{\alpha_1/m}$ is an automorphism of $V_L$ of order $m.$ It is easy to see that the group $G$ generated by $\sigma_m$ and $\tau$ is isomorphic to $D_{2m}.$ Note that
for any $h=\sigma_m^s$ with $s \not \equiv 0 \mod m$,
$$V_L(h)=\oplus_{n\geq0}V_L(h)_{\frac{1}{o(h)^2}+\frac{1}{o(h)}n}, $$
$$V_L(h\tau)=\oplus_{n\geq0}V_L(h\tau)_{\frac{1}{4}n}$$
by \cite[Theorem 5.11]{EMS}. In particular, $\theta_{V_{E-\alpha_1/m}}= e^{2 \pi i/m^2} $ and $\theta_{V_L(\sigma_m \tau)^+}= i$ in $\CC_{V_L^G}$
where $E$ is defined as in (1) and $V_L(\sigma_m \tau)^+
=\oplus_{n\geq0}V_L(\sigma_m\tau)_{\frac{1}{4}+n}.$
Therefore,   $\FSexp(\CC_{V_L^G})$ is a multiple of $4m^2$.

By Theorem \ref{t7.1} again, $(\CC_{V_L^G}, F^{V_L, G}) \cong (\ZZ(\Vec_G^\w), \iota_\w)$ as braided $\EE$-categories for some $\w \in H^3(G, S^1)$. Therefore,  $\FSexp(\ZZ(\Vec_G^\w))=\FSexp(\CC_{V_L^G})$ is a multiple of $4m^2$. By \cite{NS}, $\FSexp(\ZZ(\Vec_G^\w))$ is the least common multiple of $\ord(\w |_H) \cdot o(H)$ where $H$ runs through all the maximal cyclic subgroups of $G$. Therefore,  $\ord(\omega|_{\< \sigma_m\>})=m$ and  $\ord(\omega|_{\<  \sigma_m^i \tau\>})=2$ for some $i$.   Since $m$ is odd, $\ord(\w)=2m$ and hence $[\CC_{V_L^G}] \in \MM(\EE)$ is of order $2m$.
The proof is complete.
\end{proof}

A fusion category $\CC$ is called \emph{pointed} if $\FPdim(V)=1$ for any simple object $V$ (cf.  \cite{EGNO} for more details). For any pointed fusion category, there exists a canonical spherical structure on $\CC$ such that $\dim(V)>0$ for each nonzero object $V$, and this implies $\dim(V)=1$ if $V$ is simple. This condition on the positivity of categorical dimensions has been  assumed throughout this paper.  The set $A = \Irr(\CC)$ forms a group under the tensor product of simple objects, and $\CC$ is equivalent to $\Vec_A^\w$ for some $\w \in Z^3(A, \C^\times)$ as fusion categories. If, in addition,  $\CC$ is braided, then $A$ is abelian and  there exists a normalized 2-cochain $c: A \times A \to \C^\times$ such that the scalar $c(g,h): e(g) \ot e(h) \to  e(h) \ot e(g)$ defines a braiding on $\Vec_G^\w$. Let $\Vec_A^{(\w, c)}$ denote this braided fusion category and hence a ribbon category with the underlying spherical structure. The pair $(\w, c)$ also defines an Eilenberg-MacLane \emph{abelian} 3-cocycle, and the cohomology class $[(\w, c)]$ in the corresponding cohomology group $H^3_{ab}(A, \C^\times)$ uniquely determines the braided equivalence class of $\Vec_A^{(\w, c)}$ \cite{JS}.

Let $(A, q)$ denote the quadratic form $q: A \to \C^\times$ on $A$. Then the set $\Quad(A)$ of quadratic forms on $A$ forms a group under the pointwise multiplication. The  cohomology group $H^3_{ab}(A, \C^\times)$ is isomorphic to $\Quad(A)$  via the \emph{trace map} $[(\w, c)] \mapsto q_c$ \cite{EM1, EM2}, where $q_c(a)= c(a,a)$ for $a \in A$. The ribbon category $\Vec_A^{(\w, c)}$ is modular if and only if the quadratic form $(A,q_c)$ is  nondegenerate. 
For any quadratic form $(A, q)$,  we denote  by $\CC(A, q)$  a ribbon category $\Vec_A^{(\w, c)}$ with $q_c= q$ and so $\theta_a:=\theta_{e(a)} = q(a)$. In particular, $\CC(A, q_0)$, where $q_0(a)=1$ for all $a \in A$, is equivalent to the Tannakian category $\Rep(\hat{A})$. By \cite{NS1}, tensor equivalences of pseudounitary fusion categories preserve the canonical spherical structures. So, for any quadratic forms $(A,q)$ and $(A', q')$, the ribbon categories $\CC(A, q)$ and $\CC(A', q')$ are equivalent if and only if $(A,q)$ and $(A', q')$ are equivalent quadratic forms, i.e., there exists an isomorphism $f: A \to A'$ of groups such that $q' \circ f = q$.  

We call a rational vertex operator algebra $V$ \emph{pointed} if  every  simple module of $V$ is simple current or $\CC_V$ is pointed. In general, every lattice vertex operator algebra $V_L$ of a  positive definite even lattice $L$ is a pointed modular category given by the quadratic form $(L^\circ/L, q_L)$ where $q_L(L+\lambda) =  \theta_{V_{L+\lambda}} = e^{\pi i (\lambda, \lambda)}$.  The converse is stated in the following proposition. 
\begin{prop}\label{pointed}
 Let $\CC$ be a pointed modular category (with positive dimensions). Then $\CC \cong \CC_{V_L}$ as modular categories for some positive definite even lattice $L$. In particular, for any pointed vertex operator algebra $V$, $\CC_V \cong \CC_{V_L}$ as modular categories for some  positive definite even lattice $L$. 
\end{prop}
\begin{proof}
If $\CC$ is a pointed modular category with  positive dimensions and ribbon structure $\theta$, then $\CC \cong \CC(A, q)$ where $A$ is the abelian group $\Irr(\CC)$ under the tensor product, and a nondegenerate quadratic form $q:A \to \C^\times$ is given by $q(a) = \theta_a$.  The first multiplicative central charge $c$ of $\CC$ is an 8-th root of unity (cf. \cite[Prop.6.7 (ii)]{DLN}). Therefore, $c = e^{2 \pi i n/8}$ for some unique element $n \in \BZ_8$, called the signature of $(A,q)$. Note that $q(A) \subset S_r^1$, where $S_r^1$ is the group of all roots of unity in $\C$.  Let $\ell$ be the minimal number of generators of $A$. It follows from \cite[Corollary 1.10.2]{Ni} that there exists a positive definite even lattice $L$ with $\ell < \rank(L) \equiv n \mod 8$, and a group isomorphism  $j: A \to L^\circ/L$  such that 
$$
q(a) = e^{\pi i (j(a), j(a))} = \theta_{V_{L+j(a)}} = q_L(j(a))\,.
$$
Therefore, $(A, q)$ and $(L^\circ /L, q_L)$ are equivalent quadratic forms, and 
$$
\CC \cong \CC(A, q) \cong \CC(L^\circ /L, q_L) \cong \CC_{V_L} 
$$
as modular categories.
\end{proof}
\begin{rem} Proposition \ref{p4.3} now follows from Proposition \ref{pointed} as $\overline{\CC_{V_L}}$ is pointed.
\end{rem}

We now turn our attention to  the case when $G$ is a finite abelian group and will prove that Proposition \ref{p7.2} holds if $G$ is generated by  two elements. Set 
$$\H^{\pt}_G = \{V \in \H_G\mid V^G \text{ is pointed}\},$$  
$$
{\MM^\pt}(\EE) =\{[\CC] \in \MM(\EE) \mid \CC \text{ is pointed}\}, $$
$$
{\MM_v^\pt}(\EE)=\{[\CC_{V^G}] \in \MM_v(\EE) \mid V \in {\H_G^\pt}\}.$$
We first observe that $\H^{\pt}_G$ is closed under the tensor product of vertex operator algebras:  
\begin{lem}
Let $G$ be a finite abelian group and $\EE=\Rep(G)$. For any $U, V \in \H_G^\pt$,  $U \ot V \in {\H_G^\pt}$. Hence,  ${\MM_v^\pt}(\EE)$ is a subgroup of ${\MM_v}(\EE)$.
\end{lem}
\begin{proof} Since $V$ is holomorphic, there is a unique irreducible  $g$-twisted $V$-module $V(g)$ up to isomorphism. Recall from Section 5 that $V(g)\circ h\cong V(g)$ for any $h\in G$ as $G$ is abelian. That is, $G_{V(g)}=G$ and $V(g)$ is a $\C^{\alpha_{V(g)}}[G]$-module with decomposition 
$$V(g)=\oplus_{\lambda\in \Lambda_{V(g)}}W_{\lambda}\otimes V(g)_{\lambda}.$$
From Theorem \ref{mthm1} we know that $\{V(g)_{\lambda}\mid g\in G, \lambda\in \Lambda\}$ gives a complete list of irreducible $V^G$-modules and 
$\qdim_{V^G}(V(g)_\l)=\dim (W_{\l})\cdot [G:G_{V(g)}]\cdot\qdim_V(V(g)).$ 
Since $V^G$ is pointed, $\qdim_{V^G}(V(g)_\l)=1$. This implies  $\qdim_V(V(g))=1$,  $[G:G_{V(g)}]=1$ and  $\dim (W_{\lambda})=1$ for all $\lambda$. Thus,  $\C^{\alpha_{V(g)}}[G]$ is a commutative semisimple  algebra for $g\in G.$ Similarly,  $\C^{\alpha_{U(g)}}[G]$ is a commutative semisimple  algebra.

Identify $G\times G$ as a subgroup of $\Aut(U\otimes V)$. Then $(U\otimes V)((g, g))\cong U(g)\otimes V(g)$ and 
$\C^{\alpha_{(U\otimes V)((g, g))}}[G\times G]= \C^{\alpha_{U(g)}}[G]\otimes \C^{\alpha_{V(g)}}[G]$ 
is a commutative semisimple algebra for $g\in G.$ Regarding $G$ as a diagonal subgroup of $G\times G$ we can realize $\C^{\alpha_{(U\otimes V)((g, g))}}[G]$ as a subalgebra of 
$\C^{\alpha_{(U\otimes V)((g, g))}}[G\times G].$ Thus, $\C^{\alpha_{(U\otimes V)((g, g))}}[G]$ is a  commutative semisimple algebra of dimension $o(G).$ Using  Theorem \ref{mthm1}
again gives 
$$\qdim_{(U\otimes V)^G}(U\otimes V)(g)_\l=\dim W_{\l}[G:G_{(U\otimes V)((g, g))}]\qdim_{U\otimes V}(U\otimes V)((g, g))=1$$
for $\lambda\in \Lambda_{(U\otimes V)((g, g))}.$
That is, $(U\otimes V)(g)_\l = (U\otimes V)((g, g))_\l$ is a simple
current and $\CC_{(U\otimes V)^G}$ is pointed, as desired.
\end{proof}

We turn to understand the group $\MM_v^{\pt}(\EE)$ for any finite abelian group $G$. 
Let 
$$
H^3(G,S^1)_{\pt} = \{\a \in H^3(G, S^1)\mid \ZZ(\Vec_G^\a) \text{ is pointed}\}.
$$
The subgroup $H^3(G, S^1)_{ab}$ of $H^3(G, S^1)$ defined in  \cite[p3471]{MN1} is shown to be the same as $H^3(G,S^1)_{\pt}$ by   \cite[Corollary 3.6]{MN1} with slightly different terminology. Therefore, $H^3(G,S^1)_{\pt}$ is a subgroup of $H^3(G, S^1)$ isomorphic to $\MM^{\pt}(\EE)$.
\begin{lem} \label{l:pt_iso}
Let $G$ be a finite abelian group and $\EE=\Rep(G)$. Then 
$$H^3(G, S^1)_{\pt} \stackrel{\Phi_G}{\cong} \MM^{\pt}(\EE).$$ 
\end{lem}
\begin{proof}
Recall the isomorphism $\Phi_G : H^3(G, S^1) \to \MM(\EE)$ from Section 3. By definition, $\Phi_G(H^3(G, S^1)_{\pt})$ is a subgroup of $\MM^{\pt}(\EE)$. Conversely, suppose $(\ZZ(\Vec_G^\a), \iota)$ is a minimal modular extension of $\EE$ such that $\ZZ(\Vec_G^\a)$ is pointed. There exists $\a' \in H^3(G, S^1)$ such that $(\ZZ(\Vec_G^\a), \iota)$ is equivalent to $(\ZZ(\Vec_G^{\a'}), \iota_{\a'})$. In particular, $\ZZ(\Vec_G^{\a'})$ is pointed and hence $\a' \in H^3(G, S^1)_{\pt}$. Now, we have 
$$
\Phi_G(\a') =[(\ZZ(\Vec_G^{\a'}), \iota_{\a'})] =[(\ZZ(\Vec_G^\a), \iota)], $$
and so $\Phi_G(H^3(G, S^1)_{\pt})=\MM^{\pt}(\EE)$.
\end{proof}

Since $G$ is abelian, for any $\a=[\w] \in H^3(G, S^1)_{\pt}$, the embedding $\iota_\w: \EE \to \ZZ(\Vec_G^\w)$ is equivalent to the inclusion of quadratic forms $i_\w: (\hat G, q_0) \to (\G^\w, q_\w)$ where $\G^\w = \Irr(\ZZ(\Vec_G^\w))$ and $q_\w(x) = \theta_x$ for $x \in \G^\w$ (cf. \cite[Thm. 3.3]{JS}). By \cite[Pro. 5.2]{MN2} or \cite[Prop. 3.5 and Cor. 3.6]{MN1}, we have an exact sequence of abelian groups
$$
1 \to \hat{G} \xrightarrow{i_\w} \G^\w \to G \to 1\,,
$$
and its  corresponding cohomology class in $H^2(G, \hat{G})$ is determined by $\a$. The triple $(\G^\w, q_\w, i_\w)$, in fact, depends on the choice of representative $\w$ of $\a$, but its equivalence class does not. This will be explained in the following discussion. 

We denote by a triple $(\G, q, i)$ for any nondegenerate quadratic form $q: \G \to \C^\times$ on a finite abelian group $\G$ of order $|G|^2$ containing an isotropic subgroup isomorphic to $\hat{G}$ under a group monomorphism $i$. Let $b_q$ be the associated nondegenerate symmetric bicharacter of $\G$. For any coset $\ol x$ of $i(\hat{G})$ in $\G$ represented by $x$, $b_q(i(\chi), x) = b_q(i(\chi), x')$ for any $x' \in \ol x$ and $\chi \in \hat{G}$. There exists a unique element $g \in G$ such that $\chi(g) =b_q(i(\chi), x)$ for all $\chi \in \hat{G}$. The assignment $p: \G/i(\hat{G}) \to G$, $\ol x \mapsto g$, defines a group monomorphism. Since $o(\G/i(\hat{G})) = o(G)$, $p$ is an isomorphism and we have an exact sequence of abelian groups:
$$
1 \to \hat{G} \xrightarrow{i} \G \to G \to 1\,.
$$

Two such triples $(\G, q, i), (\G', q', i')$ are called \emph{equivalent} if there exists a group isomorphism  $j: \G \to \G'$ such that $q'\circ j = q$ and $i' = j \circ i$. Let $Q(G)$ be the set of equivalence classes  $[\G, q, i]$ of the triples $(\G, q, i)$. One can  define the product of two such triples $(\G, q,  i), (\G', q',  i')$  as follows: Let $H$ be  the subgroup $\{(i(\chi), i'(\ol{\chi})) \mid \chi \in \hat{G}\}$ of $\G \times \G'$,  and 
$$
H^\perp= \{(x, y) \in \G \times \G' \mid b(x, i(\chi)) b'(y, i'(\ol{\chi})) = 1 \text{ for all } \chi \in \hat{G}\}
$$
where $b$ and $b'$ are nondegenerate bicharacters associated with $q, q'$ respectively. Then the quadratic form $q\perp q'$ on $\G \times \G'$ induces  a nondegenerate quadratic form $q''$ on $H^\perp/H$, and $i'' (\chi) = (i(\chi), 1) H$ for $\chi \in \hat{G}$ defines an embedding of quadratic form from $\hat{G}$ into $H^\perp/H$. Then $Q(G)$ forms an abelian group with the multiplication given by 
$$
[\G, q, i]\cdot [\G', q', i'] = [H^\perp/H, q'', i'']\,,
$$
and we have the exact sequence 
$$1 \to \hat{G} \xrightarrow{i''} H^\perp/H \to G \to 1\,.
$$
The  quadratic form $ev: \hat{G} \times G \to \C^\times$, given by the evaluation map $ev$, with the embedding $i_0: \hat{G}\to  \hat{G} \times G,\, \chi \mapsto (\chi, 1)$ defines the identity class $[\hat{G} \times G, ev, i_0]$ of $Q(G)$. 

For any $\w, \w' \in Z^3(G, S^1)_{\pt}$ such that $\a = [\w]=[\w']$, the triples 
$(\G^\w, q_\w, i_\w)$ and $(\G^{\w'}, q_{\w'}, i_{\w'})$ are equivalent \cite[Prop. 3.5]{MN1}, and its equivalence class will be denoted by 
$[\G^\a, q_\a, i_\a]$. The function $\phi_G:  H^3(G, S^1)_{\pt} \to Q(G), \a \mapsto [\G^\a, q_\a, i_\a]$,  will be shown to be an isomorphism of groups.

Note that $Q(G)$ is a finite group. Any class $[\G, q, i]$ determines an equivalent class of pointed minimal modular  extension $[(\CC(\G, q), \iota)]$ of $\EE$ with the embedding $\iota: \EE \to  \CC(\G, q)$ given by $i$. Let $\psi: Q(G) \to \MM^{\pt}(\EE)$ be the mapping $[\G, q, i] \mapsto [(\CC(\G, q), \iota)]$ (cf. \cite[Thm. 3.3]{JS}). Since every pointed minimal modular  extension of $\EE$ is an image of $\psi$, $\psi$ is a bijection. It is easy  to see that $\psi$ preserves the products of these groups, and so   
we have proved the first assertion of following proposition.
\begin{prop}\label{p:Q(G)}
 The map $\psi: Q(G) \to \MM^{\pt}(\EE)$ is an isomorphism of groups, and we have the commutative diagram:
 $$
\begin{tikzcd}
H^3(G,S^1)_{\pt}  \arrow[r,"\Phi_G"]  \arrow[rd,"\phi_G"'] & \MM^{\pt}(\EE) \arrow[d, "{\psi^{-1}}"]   \\
& Q(G)  
 \end{tikzcd}
$$
Hence, $\phi_G$ is an isomorphism of groups. In particular, for any $\a, \a' \in H^3(G, S^1)_{\pt}$, we have
$$[\G^\a, q_\a, i_\a] \cdot [\G^{\a'}, q_{\a'}, i_{\a'}] =   [\G^{\a \a'}, q_{\a\a'}, i_{\a \a'}]\,.
$$
\end{prop}
\begin{proof}
The equality $
\psi^{-1} \circ \Phi_G = \phi_G
$ follows directly from the definitions of $\phi_G$, $\Phi_G$ and $\psi$,  and the fact that $\psi^{-1}([\ZZ(\Vec_G^\a), \iota_\a])= [\G^\a, q_\a, i_\a]$ for $\a \in H^3(G, S^1)_{\pt}$.  Since $\Phi_G$ and $\psi$ are isomorphisms of groups, and so is $\phi_G$. The last statement is a consequence of the commutative diagram.
\end{proof}

\begin{thm} \label{t:pt}
 Let $G$ be a finite abelian group and $\EE=\Rep(G)$. Then for any $\a \in H^3(G,S^1)_{\pt}$, there exists a positive definite even unimodular   lattice $E$ such that $V_E$ admits an automorphism group isomorphic to $G$ and 
 $(\CC_{V_E^G}, F^{V_E, G}) \cong (\ZZ(\Vec_G^\a), \iota_\a)$ as minimal modular extensions of $\EE$. Moreover, $\MM_v^\pt(\EE) = \MM^{\pt}(\EE) \cong H^3(G,S^1)_{\pt}$.
\end{thm}
\begin{proof}
Let $[\G, q, i] \in Q(G)$. By \cite[Cor. 1.10.2]{Ni}, there exists a  positive definite even lattice $L$ such that $(\G, q) \stackrel{j}{\cong} (L^\circ /L, q_L)$ for some isomorphism $j$ of quadratic forms. Let $E$ be the subgroup of $L^\circ$ containing $L$ such that $j \circ i(\hat{G})=E/L$. Then we have the following row exact commutative diagram:
$$
\begin{tikzcd}
1 \arrow[r] & \hat{G}\arrow[equal]{d} \arrow[r, "i"] \arrow[r, "i"] & \G \arrow[d, "j"] \\ 
1 \arrow[r] & \hat{G} \arrow[r, "j\circ i"] & L^\circ/L  
\end{tikzcd}\,.
$$
In particular, $[\G, q, i] = [L^\circ /L, q_L, j\circ i]$ in $Q(G)$.

For any $x \in E$, $q_L(L+x) = 1$ or $(x,x)$ is a positive even integer. Thus, $E$ is a positive definite even lattice. Since $E/L \cong \hat{G}$, $[E:L]=o(G)$ and so
$$
|\det(E)| = |\det(L)|/[E:L]^2 = o(G)^2/o(G)^2=1\,.
$$
Therefore, $E$ is unimodular.

Now, we identify $L^\circ/E$ with  $G$ via the isomorphism $p: L^\circ/E \to G$ with $p(E+x)=g$ given by $b_L(j\circ i(\chi), L+x) = \chi(g)$ for all $\chi \in \hat{G}$ where $b_L$ is the associated bicharacter of $q_L$. We consider $G$ as an automorphism group of the lattice vertex operator  algebra  $V_E = M(1) \ot \C^\e[E]$ via $p$, namely  $g=\sigma_{x_g}$ where $p(E+x_g) = g$. Then $V_E^G = V_L$, and 
$$
V_E = \bigoplus _{L+x \in E/L} V_{L+x}=\bigoplus _{\chi \in \hat{G}} V_{j\circ i(\chi)}
$$
as a $V_L$-module. So, $\irr(\EE_{V_E^G}) =\{ V_{j\circ i(\chi)} \mid \chi \in \hat{G}\}$ and $F^{V_E, G}(\chi) = V_{-j\circ i(\chi)}$.
Thus, $\psi^{-1}([\CC_{V_E^G}, F^{V_E, G}])=[L^\circ/L, q_L, i']$ where $i'(\chi) = -j \circ i (\chi)$ for $\chi \in \hat{G}$. However, $(L^\circ/L, q_L, i') \cong (L^\circ/L, q_L, j \circ i)$ under the automorphism $x \mapsto -x$ in $\Aut(L^\circ/L, q_L)$. Therefore, 
$$
\psi^{-1}([(\CC_{V_E^G}, F^{V_E, G})])=[L^\circ/L, q_L, i'] = [L^\circ/L, q_L, j \circ i] = [\G, q, i]\,,
$$
and hence $\psi^{-1}(\MM_v^{\pt}(\EE)) = Q(G)$. The remaining statement follows immediately from Lemma \ref{l:pt_iso} and Proposition \ref{p:Q(G)}.
\end{proof}
Recall from \cite[p3480]{MN1} that the group epimorphism $\varphi^*: H^3(G, S^1) \to \Hom(\bigwedge^3 G, S^1)$ defined by
$$
\varphi^*([\w])(a, b, c) = \frac{\w(a,b,c)\w(b,c,a)\w(c,a,b)}{\w(b,a,c)\w(a,c,b)\w(c, b, a)}
$$
for $a, b, c \in G$ and $\w \in Z^3(G, S^1)$. The definition of $\varphi^*$ is independent of the choice of representatives of the cohomology class $[\w]$, and its kernel was characterized in \cite[Lem. 7.4]{MN1} as
$$
H^3(G, S^1)_{\pt} = \ker \varphi^*\,.
$$
This gives us  the following corollary.
\begin{coro}\label{c8.8}
If $G$ is a finite abelian group generated by two elements, then $\MM_v(\EE) \cong H^3(G, S^1)$.
\end{coro}
\begin{proof}
Suppose $G$ is generated by $a,b$. Then, for any $x,y,z \in G$ and $\a  \in H^3(G, S^1)$, we have
$\varphi^*(\a)(x, y, z)$ is a product of the values of $\varphi^*(\a)$ at the following triples:
$$
(a, a, a),  (a, a, b),  (a, b, a), (b, a, a), (b, b, a), (b, a, b), (a, b, b), (b, b, b)\,.
$$
Since they are all equal to 1, $\a \in H^3(G, S^1)_{\pt}$ by \cite[Lem. 7.4]{MN1}. So $H^3(G, S^1)_{\pt} = H^3(G, S^1)$. Now, the result follows from Theorem \ref{t:pt}.
\end{proof}
\begin{rem} The embedding  $F^{V,G}:\EE\hookrightarrow\CC_{V^G}$ 
in $(\CC_{V^G}, F^{V,G})$ plays an essential role 
in identifying $(\CC_{V^G}, F^{V,G})$ with the corresponding $\a\in H^3(G,S^1).$ It is possible that the modular tensor categories $\CC_{V^G}$ and $\CC_{U^G}$ are braided equivalent but $(\CC_{V^G}, F^{V,G})$ and $(\CC_{U^G}, F^{U,G})$ give
two different elements in group $\MM_v(\EE).$ Or equivalently, 
there are two inequivalent embeddings $\EE\hookrightarrow\CC_{V^G}.$ The following example explains this in details. 
\end{rem}
\begin{ex}\label{e8.10}{\rm
Let $L=\Gamma_{16}$ be the spin lattice of rank 16. We now give an automorphism group $G\cong \Z_2\times \Z_2$ of $V_L$  such that
there are three inequivalent embeddings: $\EE\hookrightarrow \CC_{V_L^G}$.  The main idea is to find 
an even lattice $K$ of $L$ such that $L/K\cong \Z_2\times \Z_2$
and $K^{\circ}/K\cong \Z_4\times \Z_4$ where $K^{\circ}$ is the dual lattice of $K$ as  usual. We thank Griess Jr. for providing us with such $K.$ 

 Let $\{\e_1, \cdots, \e_{16}\}$  be the standard orthonormal basis of ${\Bbb R}^{16}.$ Recall that root lattice  $L_{D_{16}}=\sum_{i=1}^{16}\Z\alpha_i$  of type $D_{16}$ where 
$\alpha_i=\e_{i}-\e_{i+1}$ for $i=1,...,15$ and $\alpha_{16}=\e_{15}+\e_{16}.$ Then $L=L_{D_{16}}+\Z w =  L_{D_{16}}\cup (L_{D_{16}}+w)$
where $w=\frac{1}{2}(\e_1+\cdots +\e_{16}).$ Also let $u=\frac{1}{2}(\e_1+\cdots +\e_8)$ and $v=\frac{1}{2}(\e_8-\e_{16}).$
Then $(u,u)=2,$  $(v,v)=\frac{1}{2}$ and $(u,v)=\frac{1}{4}.$
Let 
$$K=\{\alpha\in L\mid (u,\alpha)\in\Z \ {\rm and}\  (v,\alpha)\in\Z\}$$
be a sublattice of   $L.$ 
It is easy to see that 
$$L/K=\{K, K+2u, K+2v, K+2u+2v\}\cong \Z_2\times \Z_2$$
and 
$$K^{\circ}=\bigcup_{i,j=0}^3(K+iu+jv), \quad K^{\circ}/K=\<u+K,v+K\>\cong \Z_4\times \Z_4.$$

Recall that $V_L=M(1)\otimes \C^{\e}[L].$ We have automorphisms $\sigma_u,\sigma_v\in \Aut(V_L).$ Then $G=\<\sigma_{u},\sigma_{v}\>$ is isomorphic to $L/K\cong \Z_2\times \Z_2.$ Moreover, 
$V_L^G=V_K$ and the irreducible $V_K$-modules are $V_{K+z},$
and 
$$
\theta_{V_{K+z}}=e^{2\pi i(s^2+\frac{1}{4}(t^2+st))} = i^{(s+t)t}
$$
for 
$z=su+tv$ with $s,t=0,...,3.$  The pair $(K^\circ/K, q)$ defines a quadratic form with $q(K+z) =\theta_{V_{K+z}}$ and $\CC_{V_L^G}$ is equivalent to the pointed modular category $\CC(K^\circ/K, q)$.
Consider the generating set $\{x, y\}$ of $K^\circ/K$ where $x = u+K$ and $y = -u +v+K$. For any $K+ z \in K^\circ/K, K+z = s x + t y = (s -t)u+ tv+K$ for some $s, t =0, \dots, 3$, and so
$$
q(K+z) = \theta_{V_{K+z}}= i^{st}\,.
$$
By direct computation,  the automorphisms of the quadratic form $(K^\circ/K, q)$ are given by
\begin{equation}\label{eq:auto}
    \Aut(K^\circ/K, q)= \{f \in \Aut(K^\circ/K) \mid q = q\circ f \} = \{\pm id, \pm \kappa\} \cong \BZ_2 \times \BZ_2
\end{equation}
where $\kappa(s x +  t y) = t x +  s y$.

Recall that the Tannakian category $\EE = \Rep(G)$ is equivalent to $\CC(\hat{G}, q_0)$ where $q_0$ is the trivial quadratic form given by $q_0=1$. Let $\psi_1, \psi_2 \in \hat{G}$ such that 
$$
\psi_1(\s_{u})=-1, \, \psi_1(\s_{v})=1 \quad \text{and}\quad   \psi_2(\s_{u})=1, \, \psi_2(\s_{v})=-1\,.
$$
Then  $F^{V_L, G}(\psi_1)= 2(x+y)$ and $F^{V_L, G}(\psi_2)=  2x$ and $F^{V_L, G}(\psi_1\psi_2) =2y$ and $F^{V_L, G}$ induces an embedding $F^{V_L, G}: (\hat{G}, q_0) \to (K^\circ/K, q)$ of quadratic forms.

Now, we twist the $G$ action on $V_L$ by an automorphism $\gamma$ of $G$, that means $g\cdot a =\gamma(g)(a)$ for $g \in G$ and $a \in V_L$, and we denote this new $G$-module by $V_L^\gamma$. Note that $(V_L^\gamma)^G = V_L^G$ as vertex operator algebras. The automorphism $\gamma$ also acts on $\hat{G}$ by composition, and $F^{V_L^\gamma, G}(\psi) = F^{V_L, G}(\psi \circ \gamma^{-1})$ for $\psi \in \hat{G}$. Thus, the corresponding embedding of quadratic forms $F^{V_L^\gamma, G}: (\hat{G}, q_0) \to (K^{\circ}/K, q)$ can be expressed as  $F^{V_L^\gamma, G}(\psi) =F^{V_L, G}(\psi \circ \gamma^{-1})$ for $\psi \in \hat{G}$.  

The automorphism group of $G$ is isomorphic to $S_3$ and each automorphism $\gamma$ is completely determined by its images of $\s_{u}$ and $\s_{v}$.    The equivalence $(\CC_{V_L^G}, F^{V_L , G}) \cong (\CC_{V_L^G}, F^{V_L^\gamma, G})$ implies the embeddings of quadratic forms $F^{V_L^\gamma, G} , F^{V_L, G} : (\hat{G}, q_0) \to (K^{\circ}/K, q)$ are equivalent, i.e., there exists automorphism $f \in \Aut(K^{\circ}/K, q)$ such that $f \circ F^{V_L, G}=F^{V_L^\gamma, G}$. By \eqref{eq:auto}, $\gamma =\id_G$ or
$$
\gamma: \s_v \mapsto \s_v , \quad \s_u \mapsto \s_u \s_v \,. 
$$
Let $\delta$ be the automorphism of $G$ given by the 3-cycle $(\s_u, \s_u \s_v, \s_v)$ in $S_3$. Then 
$(\CC_{V_L^G}, F^{V_L^\delta, G})$ and $(\CC_{V_L^G}, F^{V_L^{\delta^2}, G})$  are not equivalent to $(\CC_{V_L^G}, F^{V_L , G})$. One can further show directly from \eqref{eq:auto} that $(\CC_{V_L^G}, F^{V_L^\delta, G}) \not\cong (\CC_{V_L^G}, F^{V_L^{\delta^2}, G})$. These three inequivalent embedding of $\EE$ correspond to three different cohomology classes $\a \in H^3(G, S^1)$ such that $\ZZ(\Vec_G^\a)$ are equivalent modular tensor categories. 
}
\end{ex}
Example \ref{e8.10} also gives the following result.
\begin{prop} If $G\cong\Z_2\times \Z_{2}\times \Z_{2}$ then $\M_v(\EE)=\M(\EE)\cong H^3(G,S^1).$
 \end{prop}
\begin{proof} From the discussion before Corollary \ref{c8.8}, we see that $H^3(G,S^1)/H^3(G,S^1)_{\pt}$ is isomorphic to $\Hom(\bigwedge^3 G, S^1) \cong \BZ_2$, and so $[H^3(G,S^1): H^3(G,S^1)_{\pt}] =2$.  By Theorem \ref{t:pt}, $\M^{\pt}_v(\EE)\cong H^3(G,S^1)_{\pt}.$ It is enough to show that
there exists $[\CC_{V_L^G}]\in \M_v(\EE)$ such that $\CC_{V_L^G}$  is not pointed.  Let $V_L$ be the lattice vertex operator algebra defined in Example \ref{e8.10}. Let $\tau\in \Aut(V_L)$ such that
$$\tau(h_{i_1}(n_1)\cdots h_{i_k}(n_k)\otimes e^{\alpha})=(-1)^kh_{i_1}(n_1)\cdots h_{i_k}(n_k)\otimes e^{-\alpha}$$
for $i_1,...,i_k\in\{1,...,16\},$ $n_i<0$ and $\alpha\in L.$
Then $G\cong\<\sigma_u,\sigma_v,\tau\>$ 
and 
$$V_L^G=V_K^+=\{a\in V_K\mid \tau(a)=a\}.$$
Since $V_{K+u}$ is an irreducible
$V_K^+$-module (see \cite{ADL}) and $\FPdim(V_{K+u})=\qdim_{V_K^+}(V_{K+u})=2$, we conclude that 
$V_{K+u}$ is not a simple current and so $\CC_{V_L^G}$ is not pointed.
\end{proof}

It is worth noting that for any nonabelian group $H$ of order 8, $\ZZ(\Rep(H)) \cong \ZZ(\Vec_{H})$ is braided equivalent to some nonpointed $\ZZ(\Vec_G^\a)$ where $G=\BZ_2^3$ (cf. \cite{GMN}). Thus, there exists an embedding $\iota: \Rep(H) \to \ZZ(\Vec_G^\a)$ so that  $(\ZZ(\Vec_G^\a), \iota)$ is a minimal modular extension of   $\Rep(H)$. In general, for any finite group $A$, $\ZZ(\Vec_A^\a)$ is a minimal modular extension of any symmetric fusion subcategory $\EE$ of  $\ZZ(\Vec_A^\a)$ with $\dim(\EE) = |A|$, i.e., a Lagrangian subcategory of $\ZZ(\Vec_A^\a)$. If $\EE$ is Tannakian, then $\EE$  is braided equivalent to $\Rep(B)$ for some uniquely determined  group $B$, and $\ZZ(\Vec_A^\a)$ is braided equivalent to $\ZZ(\Vec_B^{\a'})$ for some $\a' \in H^3(B, S^1)$.\\

\noindent{{\bf Acknowledgement:}} We would like to thank Robert Griess Jr. for his crucial suggestion on Example \ref{e8.10}. The second author would also like to acknowledge that this joint work began while he was a member in residence at MSRI in the Spring of 2020 for the program on Quantum Symmetry supported by the NSF under  the  Grant  No.  DMS-1440140.

\end{document}